\documentclass{amsart}
\usepackage{amssymb,stmaryrd}
\usepackage{mathrsfs}
\usepackage[colorlinks,citecolor=blue,urlcolor=black,linkcolor=black]{hyperref}

\newtheorem{thm}{Theorem}[section]
\newtheorem{cor}[thm]{Corollary}
\newtheorem{prop}[thm]{Proposition}
\newtheorem{fact}[thm]{Fact}
\newtheorem{lemma}[thm]{Lemma}
\newtheorem{claim}{Claim}[thm]
\newtheorem{mainthm}{Theorem}

\theoremstyle{definition}
\newtheorem{definition}[thm]{Definition}
\newtheorem{q}[thm]{Question}

\theoremstyle{remark}
\newtheorem{remark}[thm]{Remark}

\newcommand*\axiomfont[1]{\textsf{\textup{#1}}}
\newcommand\ch{\axiomfont{CH}}

\newcommand\zfc{\axiomfont{ZFC}}
\newcommand\nsp{\axiomfont{NSP}}
\newcommand\TP{\axiomfont{TP}}
\newcommand\sq{\sqsubseteq}
\newcommand\s{\subseteq}
\newcommand\br{\blacktriangleright}
\renewcommand{\restriction}{\mathbin\upharpoonright}
\renewcommand\mid{\mathrel{|}\allowbreak}

\newcommand{\ssr}{\axiomfont{SSR}}
\newcommand{\mm}{\axiomfont{MM}}
\newcommand{\pfa}{\axiomfont{PFA}}
\newcommand{\mrp}{\axiomfont{MRP}}
\newcommand{\pid}{\axiomfont{PID}}
\newcommand{\ma}{\axiomfont{MA}}
\newcommand{\sch}{\axiomfont{SCH}}
\newcommand{\isp}{\axiomfont{ISP}}
\newcommand{\Add}{\mathrm{Add}}
\newcommand{\bb}{\mathbb}

\DeclareMathOperator{\add}{Add}
\DeclareMathOperator{\h}{ht}
\DeclareMathOperator{\pr}{Pr}

\DeclareMathOperator{\id}{id}
\DeclareMathOperator{\reg}{Reg}
\DeclareMathOperator{\cl}{cl}
\DeclareMathOperator{\cf}{cf}
\DeclareMathOperator{\im}{Im}
\DeclareMathOperator{\otp}{otp}
\DeclareMathOperator{\dom}{dom}
\DeclareMathOperator{\acc}{acc}
\DeclareMathOperator{\nacc}{nacc}
\DeclareMathOperator{\U}{U}
\DeclareMathOperator{\ssup}{ssup}
\DeclareMathOperator{\ind}{ind}
\newcommand\inds{\boxminus^{\ind}}
\newcommand\inde{\square_-^{\ind}}

\author{Chris Lambie-Hanson}
\address{Institute of Mathematics,
Czech Academy of Sciences,
{\v Z}itn{\'a} 25, Prague 1,
115 67, Czech Republic}
\urladdr{https://math.cas.cz/lambiehanson}

\author{Assaf Rinot}
\address{Department of Mathematics, Bar-Ilan University, Ramat-Gan 5290002, Israel.}
\urladdr{http://www.assafrinot.com}

\author{Jing Zhang}
\address{Department of Mathematics, University of Toronto
Bahen Centre, Room 6290
40 St. George St., Toronto, ON, M5S 2E4}
\urladdr{https://jingjzzhang.github.io/}

\title{Squares, ultrafilters and forcing axioms}

\begin{document}
\date{\today}
\begin{abstract}
We study relationships between various set theoretic compactness principles, focusing on
the interplay between the three families of combinatorial objects or principles mentioned in the title. Specifically, we show the following.
\begin{enumerate}
\item Strong forcing axioms, in general incompatible with the existence of indexed squares, can be made compatible with weaker versions of indexed squares.
\item Indexed squares and indecomposable ultrafilters with suitable parameters can coexist. As a consequence, the amount of stationary reflection known to be implied by the existence of a uniform indecomposable ultrafilter is optimal.
\item The Proper Forcing Axiom implies that any cardinal carrying a uniform indecomposable ultrafilter is either measurable or a supremum of countably many measurable cardinals. Leveraging insights from the preceding sections, we demonstrate that the conclusion cannot be improved.
\end{enumerate}
\end{abstract}
\maketitle
\section{Introduction}

The study of compactness and incompactness phenomena in combinatorial set theory has a long history.
On the incompactness side, the square principles ($\square$), discovered by Jensen \cite{jensen} in  his fine structural analysis of the constructible universe, have been used to settle many independent questions. Such principles make it possible to generalize techniques and proofs available at the level of the first uncountable cardinal to higher cardinals. For example, with square principles, the ``walks on ordinals" techniques discovered by Todorcevic \cite{todorcevicwalks} are available at higher cardinals, giving rise to many applications inside and outside of set theory \cite{todorcevicwalksbook}. On the compactness side, large cardinal axioms play an essential role in settling independent questions, usually in an opposite way from how square principles decide them. They are also known to directly imply statements about objects relatively low in the cumulative hierarchy; for example, \emph{Projective Determinacy} \cite{PD}. One particularly
important class of strong compactness principles, whose consistency can usually be established by performing
iterated forcing over models of large cardinals, is the class of \emph{forcing axioms}. These can be thought of as generalizations of the Baire Category Theorem in the following two aspects: they are 1) applied to more general topological spaces and 2) designed to meet more requirements/dense sets.
Two notable forcing axioms, the \emph{Proper Forcing Axiom} ($\pfa$), introduced by Baumgartner \cite{baumgartnerpfa}, and \emph{Martin's Maximum} ($\mm$), introduced by Foreman, Magidor and Shelah \cite{MM}, have found wide ranging applications both inside and outside of set theory.

In this paper, we study certain combinatorics of ultrafilters under strong forcing axioms and use a weaker version of indexed squares to demonstrate the optimality of the theorem. We need a few more definitions in order to state the main results.

\begin{definition}
An ultrafilter $U$ over an infinite cardinal $\theta$ is said to be
\begin{enumerate}
\item \emph{uniform} if $|X|=|\theta|$ for every $X\in U$;
\item \emph{weakly normal} if for any regressive $f: \theta\to \theta$, there exists $\tau<\theta$ such that $f^{-1}[\tau]\in U$.
\end{enumerate}
\end{definition}

\begin{definition}[Keisler, Prikry \cite{prikrythesis}]
Let $U$ be an ultrafilter over a set $I$, and let $\mu$ be a infinite cardinal. $U$ is
said to be  \emph{$\mu$-decomposable} if there exists a function $f : I \rightarrow \mu$ such that
$f^{-1}[H] \not\in U$ for every $H \in [\mu]^{<\mu}$.
Otherwise, it is said to be \emph{$\mu$-indecomposable}.
\end{definition}
A ultrafilter on $\kappa$ is \emph{indecomposable} if it is $\nu$-indecomposable for every $\nu\in [\aleph_1, \kappa)$. Hence, if we compare the definition with the ultrafilter given by a measurable cardinal, it is weaker in that it is possibly not countably complete. This makes it possible for non large cardinals to carry such ultrafilters.

Silver \cite{silver} asked whether a strongly inaccessible $\kappa$ carrying a uniform indecomposable ultrafilter is necessarily measurable. Sheard \cite{sheard} answered the question negatively. We give another proof of this result (see Theorem \ref{ind_sq_thm}).
However, such independent configurations cannot occur when certain structural constraints are imposed on the ground model. For example, Donder, Jensen and Koppelberg \cite{DonderJensenKoppelberg} showed that if an inaccessible $\kappa$ carries a $\mu$-indecomposable ultrafilter for some $\mu<\kappa$, then there exists an inner model of a measurable cardinal. Hence, in $L$, Silver's question has a trivial positive answer. One can show, using Kunen's analysis \cite{iteratedUltra}, in $L[\mu]$, the canonical inner model for one measurable cardinal, Silver's question also has a positive answer. It is likely such analysis generalizes to other canonical inner models.

What is more surprising is that strong large cardinals give rise to the positive answer of Silver's question as well. More recently, Goldberg \cite{goldberg} showed that any cardinal $\kappa$ carrying a uniform indecomposable ultrafilter must either be measurable or a supremum of countably many measurable cardinals provided $\kappa$ is above a strongly compact cardinal. Our first main result shows that the same conclusion follows from strong forcing axioms.

\begin{mainthm}\label{theorem: main1}
$\pfa$ implies that any cardinal carrying a uniform indecomposable ultrafilter must be either measurable or a supremum of countably many measurable cardinals.
\end{mainthm}

Goldberg's theorem and Theorem~\ref{theorem: main1} add to the long list of
combinatorial statements that were first shown to hold above a strongly compact or supercompact
cardinal and later shown to also follow from strong forcing axioms.
A popular heuristic explaining this phenomenon is that strong forcing axioms assert that $\omega_2$ behaves in many ways like a strongly compact or supercompact cardinal.
For example, Solovay showed \cite{solovay}  that, if $\kappa$ is a strongly compact cardinal, then
$\square(\lambda)$ fails for every regular cardinal $\lambda \geq \kappa $ and the Singular Cardinals Hypothesis holds above $\kappa$. Later, Todorcevic \cite{todorcevicPFA} and Viale \cite{viale}, respectively, showed the same conclusions hold with $\kappa = \omega_2$ under $\pfa$.

Next, in order to demonstrate that the conclusion we get in Theorem~\ref{theorem: main1} is optimal, we study the relationship between forcing axioms and certain indexed square principles. In what follows,
$\inds(\kappa, \theta)$ and $\inde(\kappa, \theta)$ are two natural weakenings of the indexed
square principles $\square^{\mathrm{ind}}(\kappa, \theta)$ (see Definitions \ref{inds} and
\ref{def: weakening}).

\begin{mainthm}\label{theorem: main2}
\begin{enumerate}
\item $\mm$ implies that $\inds(\kappa, \omega_1)$ fails for all regular $\kappa > \omega_1$.
\item For every pair $\theta < \kappa$ of infinite regular cardinals, there exists a $\theta^+$-directed
closed, $({<}\kappa)$-distributive forcing that adds a $\inde(\kappa, \theta)$-sequence. In
particular, $\mm$ is compatible with $\inde(\kappa, \omega_1)$ holding for all regular
$\kappa > \omega_1$.
\end{enumerate}
\end{mainthm}

Our third main result concerns the co-existence of indexed square principles and indecomposable ultrafilters. As a consequence, we show that the amount of stationary reflection implied by the existence of a uniform indecomposable ultrafilter is optimal.

\begin{mainthm}\label{theorem: main3}
Relative to the existence of a measurable cardinal, it is consistent that $\square^{\mathrm{ind}}(\kappa,\theta)$ holds and $\kappa$ carries a uniform ultrafilter
that is $\mu$-indecomposable  for every cardinal $\mu\in [\theta^+, \kappa)$.
\end{mainthm}

As a corollary to Theorem \ref{theorem: main2} and the proof of Theorem \ref{theorem: main3}, we will show
in Theorem \ref{theorem: optimal} that $\mm$ (and hence $\pfa$) is
compatible with the existence of a strongly inaccessible cardinal that is not weakly compact but
carries a uniform ultrafilter that is $\mu$-indecomposable for every cardinal $\mu \in [\aleph_2,
\kappa)$, thus demonstrating the optimality of Theorem \ref{theorem: main1}.

\subsection{Organization of this paper}
In Section~\ref{section: trees}, we give a brief overview of an important technique of Kunen \cite{Kunen}
and then use variations of this technique to answer several questions in the literature regarding trees.

In Section~\ref{section: forcingaxiomssquare}, we introduce various indexed square
principles and prove Theorem \ref{theorem: main2}. We also answer a question from \cite{hlh}
by showing that $\square(\kappa, \theta)$ does not in general imply the existence of a
\emph{full} $\square(\kappa, \theta)$-sequence.

In Section~\ref{section: ultrafiltersquare}, we investigate the effect of indecomposable ultrafilters
on a variety of combinatorial principles, including the C-sequence number, trees with ascent paths,
strong colorings, and square principles. We prove Theorem \ref{theorem: main3} and apply similar techniques to
reproduce consistency results concerning partially strongly compact cardinals.

In Section~\ref{section: FAandUltra}, we prove Theorem \ref{theorem: main1} and then use results from
Sections \ref{section: forcingaxiomssquare} and \ref{section: ultrafiltersquare} to establish its
optimality.

Finally in Section~\ref{Section: questions}, we conclude with some open questions.

\subsection{Notation and conventions}\label{conventions}
$\reg(\kappa)$ stands for set of all infinite regular cardinals below $\kappa$.
For a set $X$, we write $[X]^{\kappa}$ for the collection of all subsets of $X$ of size $\kappa$. The collections $[X]^{\le\kappa}$ and $[X]^{<\kappa}$ are defined similarly.
For a set of ordinals $A$, we write $\ssup(A) := \sup\{\alpha + 1 \mid
\alpha \in A\}$, $\acc(A) := \{\alpha\in A \mid \sup(A \cap \alpha) = \alpha > 0\}$,
$\nacc(A) := A \setminus \acc(A)$, and $\acc^+(A) := \{\alpha < \ssup(A) \mid \sup(A \cap \alpha) =
\alpha > 0\}$.

If $a$ and $b$ are sets of ordinals, then $a < b$ is the assertion that $\alpha < \beta$ for
all $\alpha \in a$ and $\beta \in b$.
If $A$ is a set of ordinals, then we write $(\alpha,\beta) \in [A]^2$ to assert that
$\alpha,\beta \in A$ and $\alpha < \beta$. If $\mathcal{A}$ is a collection of sets of ordinals,
then we write $(a,b) \in [\mathcal{A}]^2$ to assert that $a,b \in \mathcal{A}$ and $a < b$.
If $a$ and $b$ are sets of ordinals, then we write $a \sqsubseteq b$ to denote
the assertion that $b$ is an end-extension of $a$. If $\delta$ is an ordinal and $\theta$ is an infinite
cardinal, then $E^\delta_\theta := \{\alpha < \delta \mid \cf(\alpha) = \theta\}$. Variations such
as $E^\delta_{\neq \theta}$, $E^\delta_{> \theta}$, etc.~are defined in the obvious way.

For a tree $(T,<_T)$ and an ordinal $\alpha$, we denote by $T_\alpha$ the $\alpha^{\text{th}}$-level of the tree,
and we write $T\restriction\beta$ for $\bigcup_{\alpha<\beta}T_\alpha$.
Also, for a pair of ordinals $\alpha<\beta$ and a node $t\in T_\alpha$, we write $t\restriction\beta$ for the unique $s<_Tt$ belonging to $T_\beta$. Given $s \in T$, we let $s^\uparrow$ denote the cone of
$T$ above $s$, i.e., the tree with underlying set $\{t \in T \mid s \leq_T t\}$, ordered by the
restriction of $<_T$.

\section{Trees at strongly inaccessible cardinals}\label{section: trees}

\subsection{A brief survey of Kunen's method}
A central concern of this paper, and of the study of combinatorial set theory more broadly, is
the determination of any causal implications that may exist among various compactness
principles. One half of this endeavor involves the task of \emph{separating} certain compactness
principles, i.e., proving that one does not imply another. In \cite{Kunen}, Kunen introduced
a useful technique for achieving such results that has been further deployed and refined
by a number of researchers in the intervening years. Since many of our results in this paper both
are directly motivated by this prior work and rely themselves on variations of Kunen's technique,
we thought it appropriate to begin this paper with a brief overview of technique and some of its
relevant applications over the last almost half century.

We will typically be interested in compactness principles that can hold at some given
cardinal $\kappa$. In light of this, we will often, e.g., let $\Phi$ denote the general
formulation of a compactness principle and let $\Phi(\kappa)$ denote an instance of $\Phi$
at a particular cardinal $\kappa$. For example, $\Phi$ could be ``the tree property", in
which case $\Phi(\kappa)$ would be ``the tree property at $\kappa$".
In broad strokes, Kunen's technique can now be summarized as follows.
Suppose that $\Phi$ and $\Psi$ are two compactness principles, and one wants to prove
that $\Phi(\kappa)$ does not imply $\Psi(\kappa)$. In a typical application,
one begins in a model $V$ of $\zfc$ with
a cardinal $\kappa$ such that $\Phi(\kappa)$ holds and is indestructible under forcing
with $\Add(\kappa, 1)$, the forcing to add a Cohen subset to $\kappa$.
One then designs a two-step forcing iteration $\bb{P} \ast \dot{\bb{Q}}$ such that
\begin{enumerate}
\item forcing with $\bb{P}$ introduces a counterexample to $\Psi(\kappa)$;
\item $\bb{P} \ast \dot{\bb{Q}}$ is forcing equivalent to $\Add(\kappa, 1)$;
\item in $V^{\bb{P}}$, forcing with $\bb{Q}$ provably preserves counterexamples to $\Phi(\kappa)$,
i.e., if $\Phi(\kappa)$ fails in $V^{\bb{P}}$, then it continues to fail in $V^{\bb{P} \ast \dot{\bb{Q}}}$.
\end{enumerate}
Clause (1) implies that $\Psi(\kappa)$ fails in $V^{\bb{P}}$, clause (2) and our initial assumption
about $\kappa$ implies that $\Phi(\kappa)$ holds in $V^{\bb{P} \ast \dot{\bb{Q}}}$, and then clause
(3) implies that $\Phi(\kappa)$ holds in $V^{\bb{P}}$. In particular, we have proven that
$\Phi(\kappa)$ does not imply $\Psi(\kappa)$, modulo the consistency of our original assumptions.

Kunen originally developed this technique in \cite[\S 3]{Kunen} to prove that an inaccessible cardinal
$\kappa$ carrying a nontrivial, $\kappa$-complete, $\kappa$-saturated ideal need not be measurable.
To give a sketch of his proof, we need to recall the following definitions, which will
continue to be relevant throughout this section.

\begin{definition}
Let $\alpha$ be an ordinal. We say that a tree $T\s {}^{<\alpha}2$ where the tree order is the natural end-extension is
\begin{itemize}
\item \emph{normal} if every for all $\gamma<\beta<\alpha$, for every node $t\in T_\gamma$, there exists a node $s\in T_\beta$ extending $t$;
\item \emph{splitting} if every node $t\in T$ admits two immediate extensions in $T$;
\item \emph{homogeneous} if for every $s\in T$, $T_s:=\{s'\mid s{}^\smallfrown s' \in T\}$ is equal to $T$.
\end{itemize}
\end{definition}

Note that, if $T \s {}^{<\alpha}2$ is a homogeneous tree, then $\alpha$ is necessary an
additively indecomposable ordinal. We will sometimes need the following slight abuse of terminology.

\begin{definition}
Suppose that $\alpha$ is indecomposable and $T \s {}^{<\alpha + 1}2$ is a normal tree. We
say that $T$ is homogeneous if, for every $s \in T_{<\alpha}$, $T_s = T$.
\end{definition}

We can now sketch a proof of Kunen's result as follows. Begin in a model $V$ of $\zfc$
in which $\kappa$ is a measurable cardinal that is
indestructible under forcing with $\Add(\kappa, 1)$. Then let $\bb{P}$ be the forcing
consisting of all normal, splitting, homogeneous strees of height $\alpha + 1$ for some
indecomposable $\alpha < \kappa$. $\bb{P}$ is ordered by end-extension, i.e., if
$p,q \in \bb{P}$, then $q \leq_{\bb{P}} p$ iff $q \restriction \mathrm{ht}(p) = p$.
One can then argue that $\bb{P}$ is $({<}\kappa)$-distributive and, in $V^{\bb{P}}$, the
union of the $\bb{P}$-generic filter is a homogeneous $\kappa$-Souslin tree $T$.
Thus, in $V^{\bb{P}}$, $\kappa$ is an inacessible cardinal that is not weakly compact,
let alone measurable. In $V$, let $\dot{T}$ be the canonical $\bb{P}$-name for this
generic $\kappa$-Souslin tree, considered as a forcing notion (the forcing order is the
reverse of the tree order). One then proves that the two-step iteration
$\bb{P} \ast \dot{T}$ has a dense $\kappa$-directed closed subset of cardinality $\kappa$
and is therefore forcing equivalent to $\mathrm{Add}(\kappa, 1)$.
By assumption, $\kappa$ is measurable in $V^{\bb{P} \ast \dot{T}}$ and hence
carries a nontrivial, $\kappa$-complete, $\kappa$-saturated ideal in that model.
Since $\dot{T}$ is forced to have the $\kappa$-cc in $V^{\bb{P}}$, the following fact,
whose proof we leave to the reader, will complete the proof.

\begin{fact}
Suppose that $\kappa$ is a regular uncountable cardinal, $\bb{Q}$ is a $\kappa$-cc
forcing notion, and $\dot{I}$ is a $\bb{Q}$-name for a nontrivial, $\kappa$-complete,
$\kappa$-saturated ideal over $\kappa$. Then
\[
J := \{X \s \kappa \mid \Vdash_{\bb{Q}} \check{X} \in \dot{I}\}
\]
is a nontrivial, $\kappa$-complete, $\kappa$-saturated ideal in $V$.
\end{fact}

A few years later, a variation on Kunen's method was employed by Sheard \cite{sheard} to prove
that an inaccessible cardinal carrying a uniform indecomposable ultrafilter
need not be measurable, answering a question of Silver. Sheard forces with a slight variation
on Kunen's forcing over the canonical inner model $L[\mu]$, where $\mu$ is a measure over
$\kappa$, to add a homogeneous $\kappa$-Souslin tree $T$. The desired model is then
$L[T, \mathscr{U}]$ where $\mathscr{U}$ is a filter over $\kappa$ in a certain further
forcing extension that becomes the desired indecomposable ultrafilter in $L[T,\mathscr{U}]$.

Because of its relevance to the results of this paper, we end this subsection by recalling
one more recent application of Kunen's method. In \cite{hlh}, building on work of Cummings,
Foreman, and Magidor \cite{cfm}, Hayut and Lambie-Hanson investigated the interplay between
$\square(\kappa, \theta)$-sequences and stationary reflection principles. For instance,
they showed that, if one starts wtih regular cardinals $\theta < \kappa$ such that $\kappa$
is weakly compact and indestructible under forcing with $\Add(\kappa, 1)$, then one can force with
a poset $\bb{P}$ to add a $\square^{\mathrm{ind}}(\kappa, \theta)$-sequence\footnote{See
Section~\ref{section: forcingaxiomssquare} for the definition of $\square^{\mathrm{ind}}(\kappa, \theta)$.}
in such a way that any of the forcings to add a thread through the generic
$\square^{\mathrm{ind}}(\kappa, \theta)$-sequence would resurrect the weak compactness of $\kappa$.
They then leveraged this fact to show that, in $V^{\bb{P}}$, every collection of fewer than
$\theta$-many stationary subsets of $\kappa$ reflects simultaneously. This is sharp, since
$\square^{\mathrm{ind}}(\kappa, \theta)$ implies the existence of a collection of
$\theta$-many stationary subsets of $\kappa$ that does not reflect simultaneously.

We shall see in Subsection \ref{subsection: ind_square_ind_uf} that forcing to add a
$\square^{\mathrm{ind}}(\kappa, \omega)$-sequence over an indestructibly measurable cardinal
yields a model in which $\kappa$ carries a uniform indecomposable ultrafilter, thus providing an
alternate proof of Sheard's result mentioned above.

\subsection{Souslin tree and diamond at an inaccessible cardinal}
Kunen proved that if $\diamondsuit(S)$ fails in $V$, where $S$ is a stationary subset of a successor cardinal $\kappa$, then it continues to fail in any further $\kappa$-cc forcing extension.
The next result shows that this is not true for $\kappa$ inaccessible.

Note that the techniques in \cite{ShelahNotCollapsing} can be used to build a model where $\kappa$ is an inaccessible cardinal, $\diamondsuit(S)$ fails for some stationary $S\subset \kappa$, and there exists a $\kappa$-Souslin tree. To see this, we can start with $L$ being the ground model with an inaccessible non-weakly compact cardinal $\kappa$. Pick some non-reflecting stationary $S\subset \kappa$ such that whose complement $S^c$ is fat. Let $E\subset S^c$ be a stationary such that $\diamondsuit(E)$ and $\square(E)$ (in the sense of \cite[Theorem 6.1]{jensen}) both hold. By \cite{jensen}, $\diamondsuit(E)$ and $\square(E)$ implies the existence of a $\kappa$-Souslin tree.
Then the forcing in \cite{ShelahNotCollapsing} giving rise to $\neg \diamondsuit(S)$ is $S^c$-closed. In particular, it preserves the stationarity of $E$, $\diamondsuit(E)$ and $\square(E)$. So there exists a $\kappa$-Souslin tree in the forcing extension.

\begin{prop}\label{lemma: Souslin} Suppose that $\kappa$ is a strongly inaccessible cardinal and there exists a $\kappa$-Souslin tree. Then in some $\kappa$-cc forcing extension, $\diamondsuit(S)$ holds for all stationary $S\s\kappa$.
\end{prop}

\begin{proof} By a standard fact (see \cite[Lemma~2.4]{paper20}), we may fix a normal $\kappa$-Souslin tree $T\s{}^{<\kappa}\kappa$ such that, for every $\delta<\kappa$ and $t\in T_\delta$, $\{ t{}^\smallfrown\langle i\rangle\mid i<2^\delta\}\s T$.
Clearly, $\mathbb P:=(T,{\supseteq})$ is a $\kappa$-cc notion of forcing. Let $G$ be $\mathbb P$-generic over $V$,
so that $g:=\bigcup G$ is a branch through $(T,{\s})$.

In $V$, for each $\delta<\kappa$, let $\langle x^\delta_i\mid i<2^\delta\rangle$ enumerate $\mathcal P(\delta)$.
In $V[G]$, let $S\s\kappa$ be a stationary set,
and we shall define a $\diamondsuit(S)$-sequence $\langle A_\delta\mid\delta\in S\rangle$, as follows.
Given $\delta\in S$, let $\epsilon\in [\delta,\kappa)$ be the least such that $g\restriction \epsilon\Vdash \delta\in \dot{S}$, and then define
$$A_\delta := \begin{cases}x^\delta_{g(\epsilon)},&\text{if }g(\epsilon)<2^\delta;\\
\emptyset,&\text{otherwise}.\end{cases}$$

We verify that this works by running a standard density argument back in $V$.
Given $t\in T$, a $\mathbb P$-name $\dot{X}$  for a subset of $\kappa$ and a club $C\s\kappa$ (in $V$), we need to find an extension $t'$ of $t$ and some $\delta\in C$ such that $t'\Vdash \delta\in \dot{S}$ and $t'\Vdash \dot{X}\cap \delta = \dot A_\delta$.

Let $\langle M_\delta\mid \delta<\kappa\rangle$ be an $\in$-increasing continuous sequence of elementary submodels
of $H(\kappa^+)$ containing $\{\mathbb P,\dot S,\dot X\}$.
Consider the club $D:=\{\delta\in C\mid M_\delta\cap\kappa= \delta\}$.
Notice that an immediate consequence of the $\kappa$-cc-ness of $\mathbb P$ gives that for every $\delta\in D$, any node $s\in T_{\delta}$ is $\mathbb P$-generic over $M_\delta$.
In addition, $\mathbb P$ is ${<}\kappa$-distributive, thus, for every $\delta\in D$, any node $s\in T_{\delta}$ decides $\dot X$ up to $\delta$.

Now, since $\dot S$ is a $\mathbb P$-name for a stationary subset of $\kappa$, we may pick some $\delta\in D$ and an extension $t^*$ of $t$ such that $\dom(t^*) \geq \delta$ and $t^*\Vdash \delta\in \dot{S}$.
Set $\epsilon:=\dom(t^*)$. By possibly going to an initial segment of $t^*$, we may assume that
$\epsilon$ is the least ordinal $\varepsilon \geq \delta$ such that $t^* \restriction \varepsilon
\Vdash \delta \in \dot{S}$.

Now pick $i < 2^\delta$ such that $t^* \restriction \delta \Vdash \dot{X} \cap \delta = x^\delta_i$.
Then $t':= t^*{}^\smallfrown \langle i\rangle$ is an extension of $t$ in $T$ such that $\epsilon$ is the least element of $[\delta,\kappa)$ to satisfy $t'\restriction\epsilon\Vdash \delta\in \dot{S}$,
and it is the case that $t'\Vdash \dot{X}\cap \delta = x^\delta_{g(\epsilon)}=\dot A_\delta$.
\end{proof}

\subsection{A non-coherent variation}
Recall that, for a regular uncountable cardinal $\kappa$, a \emph{$C$-sequence over $\kappa$} is
a sequence $\langle C_\beta \mid \beta < \kappa \rangle$ such that, for all $\beta < \kappa$,
$C_\beta$ is a closed subset of $\beta$ with $\sup(C_\beta) = \sup(\beta)$.

In \cite{knaster_ii}, a measure $\chi(\kappa)$ for a cardinal $\kappa$ was introduced to describe how far it is from being weakly compact.
If $\kappa$ is weakly compact, then $\chi(\kappa):= 0$. Otherwise, $\chi(\kappa)$ denotes the least cardinal $\chi\leq \kappa$ such that for every $C$-sequence $\langle C_\beta\mid \beta<\kappa\rangle$, there exist $\Delta\in [\kappa]^\kappa$ and $b: \kappa\to [\kappa]^{\chi}$ such that $\Delta\cap \alpha \s \bigcup_{\beta\in b(\alpha)} C_{\beta}$ for every $\alpha<\kappa$. The cardinal $\chi(\kappa)$ is
referred to as the \emph{$C$-sequence number of $\kappa$}.
Question~6.4 of the same paper asks whether a strongly inaccessible cardinal $\kappa$ satisfying $\chi(\kappa) = 1$
must admit a coherent $\kappa$-Aronszajn tree.
As a coherent $\kappa$-Aronszajn tree cannot contain a copy of the tree ${^{\leq \omega}}2$,
the following theorem answers the above question in the negative. We first recall the important notion
of \emph{strategic closure}.

\begin{definition}
Let $\bb{P}$ be a partial order (with maximum element $1_{\bb{P}}$) and let $\beta$ be
an ordinal.
\begin{enumerate}
\item $\Game_\beta(\bb{P})$ is the two-player game in which Players I and II alternate
playing conditions from $\bb{P}$ to attempt to construct a $\leq_{\bb{P}}$-decreasing
sequence $\langle p_\alpha \mid \alpha < \beta \rangle$. Player I plays at odd stages, and
Player II plays at even stages (including limit stages). Player II is required to play
$p_0 = 1_{\bb{P}}$. If, during the course of play, a limit ordinal $\alpha < \beta$ is
reached such that $\langle p_\eta \mid \eta < \alpha \rangle$ has no lower bound in
$\bb{P}$, then Player I wins. Otherwise, Player II wins.
\item $\bb{P}$ is said to be \emph{$\beta$-strategically closed} if Player II has a winning
strategy in $\Game_\beta(\bb{P})$.
\end{enumerate}
\end{definition}

We will often speak about strategic closure of a poset $\bb{P}$ in which we have not explicitly
added a maximum element $1_{\bb{P}}$. In this case, we implicitly add $\emptyset$ as a maximum
condition to $\bb{P}$.
Note that, if $\kappa$ is a regular cardinal and $\bb{P}$ is $\kappa$-strategically closed, then
$\bb{P}$ is $({<}\kappa)$-distributive.

\begin{thm}\label{theorem: CantorSubtree} Suppose that $\kappa$ is weakly compact. Then there is a ${<}\kappa$-distributive forcing extension in which
$\chi(\kappa)=1$
and every $\kappa$-Aronszajn tree contains a copy of ${}^{\leq\theta}2$ for every $\theta<\kappa$.
\end{thm}
\begin{proof}
We will construct a model with a Souslin tree $T$ such that
\begin{itemize}
\item $\Vdash_T \kappa$ is weakly compact,
\item for every $\theta<\kappa$ and every $x\in$ with $\h(x) > \theta$, $x^\uparrow$ is $\theta^+$-closed.
\end{itemize}

Consider the forcing $\mathbb P_\kappa$ consisting of all conditions $t$ such that:
\begin{itemize}
\item $t$ is a normal, splitting, homogeneous tree of height $\alpha+1$ for some $\alpha<\kappa$,
\item $t$ is closed at singular levels: for every singular cardinal $\gamma\le\alpha$, and every $<_t$-increasing sequence $\langle s_i\mid i<\gamma\rangle$ of nodes below level $\gamma$,
there is a node $s$ in $t$ such that $s_i<_t s$ for every $i<\gamma$.
\end{itemize}

The order is end-extension.

\begin{claim}
$\mathbb P_\kappa$ is $\kappa$-strategically closed.
\end{claim}
\begin{proof}
The strategy for Player~II is simply to continue all cofinal branches.
\end{proof}

\begin{claim}
$\mathbb P_\kappa$ adds a $\kappa$-Souslin tree.
\end{claim}
\begin{proof}
Let $\dot{T}_\kappa$ be the canonical name for the union of the $\bb{P}_\kappa$-generic filter,
let $p \in \bb{P}$, and let $\dot{X}$ be a $\bb{P}$-name such that
$p\Vdash ``\dot{X}$ is a maximal antichain in $\dot{T}_\kappa"$. Let $\chi$ be a sufficiently large regular cardinal, and find an elementary submodel $M\prec H(\chi)$ containing all relevant objects such that $M\cap \kappa=\lambda\in\reg(\kappa)$ and ${}^{<\lambda}M\subset M$. Such $\lambda$ exists since $\kappa$ is Mahlo. Using the strategy for Player~II, we can define a decreasing sequence $\langle t_i\mid i<\lambda\rangle$
from $M \cap \bb{P}_\kappa$ that is $(M, \bb{P}_\kappa)$-generic, i.e., it meets every dense subset $D \subseteq \bb{P}_\kappa$ such that $D \in M$. Let $t'=\bigcup_{i<\lambda} t_i$. By the genericity, there exists a maximal antichain $X'\subset t'$ such that any condition extending $t'$ forces that $\dot{X}\cap \lambda =X'$. Then we just continue certain cofinal branches through $t'$ at level $\lambda$ to seal $X'$. Namely, each branch has to pass through a node in $X'$. We need to maintain the normality as well as the homogeneity of the tree also but this is easy since $X'$ is a maximal antichain of $t'$. The reader is referred to \cite[Pages 70--71]{Kunen} for more details. The key point here is that since $\lambda$ is a regular cardinal, we are not obliged to complete all branches. As a result, the ``antichain sealing" is possible.
\end{proof}

Let $T_\kappa$ be the $\kappa$-Souslin tree added by $\mathbb P_\kappa$. Then, in $V$, $\mathbb P_\kappa*T_\kappa$ has the following dense subset: $\{(t, b)\in \bb{P}_\kappa \times {^{<\kappa}}2 \mid  \h(t)=\dom(b)+1 \text{ and } t\Vdash b\in \dot{T}_\kappa\}$. This dense set is ${<}\kappa$-closed of size $\kappa$. Hence it is forcing equivalent to $\add(\kappa,1)$.

The final model is obtained by performing an Easton-support iteration with $\mathbb P=\langle \mathbb R_\alpha\mid \alpha<\kappa\rangle$ followed by $\mathbb P_\kappa$, where $\mathbb R_\alpha:=\add(\alpha,1)$ for $\alpha$ Mahlo and trivial otherwise. In the final model, we have a $\kappa$-Souslin tree such that forcing with it restores the weak compactness of $\kappa$. Furthermore, this Souslin tree is closed at singular levels. In particular, given $\theta<\kappa$, for every $x\in T$ with $\h(x) > \theta$, $x^\uparrow$ is $\theta^+$-closed. Since every $\kappa$-Aronszajn tree in $V[\mathbb P*\mathbb P_\kappa]$ obtains a cofinal branch in some $\theta^+$-closed forcing extension,
it follows that it must contain a copy of ${}^{\leq\theta}2$.
\end{proof}

\subsection{A star variation}
For a $C$-sequence $\langle C_\beta\mid \beta<\kappa\rangle$,
$\chi(\vec C)$ stands for least cardinal $\chi\leq \kappa$ such that
there exist $\Delta\in [\kappa]^\kappa$ and $b: \kappa\to [\kappa]^{\chi}$ such that $\Delta\cap \alpha \subset \bigcup_{\beta\in b(\alpha)} C_{\beta}$ for every $\alpha<\kappa$.
By \cite[Lemma~4.12]{knaster_ii}, if $\vec C$ witnesses $\square(\kappa)$, then $\chi(\vec C)=\sup(\reg(\kappa))$.
Here, we point out that this cannot be weakened to the following principle $\square(\kappa,{\sq^*})$.
\begin{definition}\label{def27}
Suppose that $\kappa$ is a regular uncounctable cardinal.
\begin{enumerate}
\item For $C,D \in \mathcal{P}(\kappa)$, we say that $C \sqsubseteq^* D$ if there is
$\gamma < \sup(C)$ such that $D \setminus \gamma$ end-extends $C \setminus \gamma$.
\item $\square(\kappa, {\sqsubseteq^*})$ is the assertion that there is a sequence
$\vec{C} = \langle C_\alpha \mid \alpha \in \acc(\kappa)\rangle$ such that
\begin{enumerate}
\item for all $\alpha \in \acc(\kappa)$, $C_\alpha$ is a club in $\alpha$;
\item $\vec{C}$ is $\sqsubseteq^*$-coherent, i.e., for all $\beta \in \acc(\kappa)$ and all
$\alpha \in \acc(C_\beta)$, we have
$C_\alpha \sqsubseteq^* C_\beta$;
\item there is no club $D$ in $\kappa$ such that, for all $\alpha \in \acc(D)$, we have
$C_\alpha\sq^* D$.
\end{enumerate}
\end{enumerate}
\end{definition}

\begin{prop} Suppose that $\kappa$ is a regular uncountable cardinal and $\square(\kappa)$ holds.
Then there is a $\square(\kappa,\sq^*)$-sequence $\vec C$ with $\chi(\vec C)=1$.
\end{prop}
\begin{proof} Given two functions $f$ and $g$, let $f =^* g$ denote the assertion that
the set $\{a \in \dom(f) \cap \dom(g) \mid f(a) \neq g(a)\}$ is finite. By \cite[Theorem~3.9]{MR2013395}, $\square(\kappa)$ yields a sequence $\langle f_\beta:\beta\rightarrow2\mid \beta<\kappa\rangle$ such that:
\begin{itemize}
\item  for all $\alpha<\beta<\kappa$, $f_\alpha =^* f_\beta$;
\item there exists no $f:\kappa\rightarrow 2$ such that, for all $\alpha < \kappa$, $f_\alpha
=^* f$.
\end{itemize}
Now, for every $\beta\in\acc(\kappa)$, let $C_\beta:=\acc(\beta)\cup\{\xi+1\mid f_\beta(\xi)=1\}$.
Then $\vec C:=\langle C_\beta\mid\beta<\kappa\rangle$ is a $\sq^*$-coherent $C$-sequence. Moreover,
$\chi(\vec C)=1$, as witnessed by $\Delta = \acc(\kappa)$. However, there is no club $D \subseteq
\kappa$ such that $C_\beta \sqsubseteq^* D$ for all $\beta \in \acc(D)$. To see this, suppose
for the sake of contradiction that $D$ is such a club.

Using the pressing-down lemma, fix a stationary set $S \subseteq \acc(D)$ and an ordinal
$\gamma < \kappa$ such that, for all $\alpha \in S$, we have $C_\alpha \setminus \gamma =
D \cap [\gamma, \alpha)$. Now define a function $f:\kappa \rightarrow 2$ by letting
$f \restriction \gamma = f_\gamma$ and, for all $\xi \in [\gamma,\kappa)$, setting
$f(\xi) = 1$ if and only if $\xi + 1 \in D$.

We will reach a contradiction by showing that $f_\alpha =^* f$ for all $\alpha < \kappa$.
To this end, fix an $\alpha < \kappa$. Note first that $f_\alpha \restriction \gamma =^* f_\gamma =
f \restriction \gamma$. Next, fix $\beta \in S \setminus \alpha$, and note that
$f_\alpha \restriction [\gamma, \alpha) =^* f_\beta \restriction [\gamma, \alpha)$.
Moreover, by the choice of $S$, we have $C_\beta \cap [\gamma, \alpha) = D \cap [\gamma, \alpha)$,
so, by the definition of $C_\beta$ and of $f$, we have $f_\beta \restriction [\gamma, \alpha) =
f \restriction [\gamma, \alpha)$. Altogether, this yields $f_\alpha =^* f$ and the desired
contradiction.
\end{proof}

\section{Forcing axioms and indexed squares}\label{section: forcingaxiomssquare}

The principle $\square^{\ind}(\kappa, \theta)$ was introduced in \cite{narrow_systems},
and it is the strengthening of the following principle obtained by requiring that $\Gamma$ be
the whole of $\acc(\kappa)$.

\begin{definition}[{\cite[\S4]{paper36}}]\label{inds}
$\inds(\kappa, \theta)$ asserts the existence of a matrix
\[
\vec{C} = \langle C_{\alpha,i}\mid \alpha\in\Gamma, ~ i(\alpha) \le i<\theta\rangle
\]
satisfying the following requirements:
\begin{enumerate}
\item $(E^\kappa_{\neq \theta} \cap \acc(\kappa)) \s \Gamma \s \acc(\kappa)$;
\item for all $\alpha\in\Gamma$, we have $i(\alpha) < \theta$, and $\langle C_{\alpha,i}\mid
i(\alpha) \le i<\theta\rangle$ is a $\s$-increasing sequence of clubs in $\alpha$, with
$\Gamma \cap \alpha = \bigcup_{i(\alpha)\le i<\theta}\acc(C_{\alpha,i})$;
\item $\vec{C}$ is \emph{coherent}, i.e., for all
$\alpha\in\Gamma$, $i(\alpha) \le i<\theta$, and $\bar\alpha \in
\acc(C_{\alpha,i})$, we have $i(\bar\alpha) \le i$ and $C_{\bar\alpha,i}=C_{\alpha,i} \cap\bar\alpha$;
\item \label{clause4} $\vec{C}$ is \emph{nontrivial}, i.e., for every club
$D$ in $\kappa$, there exists $\alpha\in\acc(D) \cap \Gamma$ such that,
for all $i<\theta$, $D\cap\alpha\neq C_{\alpha,i}$.
\end{enumerate}
\end{definition}
\begin{remark} $\square^{\ind}(\kappa, \theta)\implies\inds(\kappa, \theta)$.
The two coincide whenever $\theta=\omega$ or assuming that $\theta\in\reg(\kappa)$ and every stationary subset of $E^\kappa_\theta$ reflects (see \cite[Theorems~4.6 and Corollary~4.7]{paper36})
\end{remark}

It is well known that Martin's Maximum (\mm) is compatible with $\square^{\ind}(\kappa,\omega_2)$ holding for all regular $\kappa\geq \omega_2$.
Related to this is a result of L\"{u}cke \cite[Theorem 5.8]{lucke} implying that $\mm$
is compatible with the existence of an $\aleph_3$-Souslin tree admitting an $\aleph_1$-ascent path.

It is known that certain fragments of $\pfa$ or $\mm$ imply $\neg \square(\kappa,\omega_1)$ for a regular $\kappa>\omega_2$ (see for example  \cite{Strullu} and \cite{TPWu}).
However, these fragments are not compatible with $\ch$. For example, \cite{Strullu} uses $\mrp + \ma$ and \cite{TPWu} uses $\ssr + \neg \ch$.

Here, we shall show that $\inds(\kappa,\omega_1)$ is ruled out by the Semi-Stationary Reflection Principle ($\ssr$), a 2-cardinal stationary reflection principle \cite[Chapter XIII, 1.7]{properShelah} that follows from $\mm$ \cite{MM} and \emph{Rado's Conjecture} \cite{Doebler} (see also \cite[Theorem 5.2]{RCzhang}),
but is also compatible with $\ch$. This corrects a claim made in \cite[Remark~4.12]{paper36} that $\mm$ is compatible with $\inds(\kappa,\omega_1)$.
Whether $\ssr$ is compatible with $\square(\kappa, \omega_1)$ for $\kappa>\omega_2$ remains an open question (see Section \ref{Section: questions}).

\begin{definition}[\cite{SakaiVelickovic}]
For countable subsets $x, y$ of a regular cardinal $\lambda\geq \omega_2$,\footnote{The definition here is unrelated to that of Definition~\ref{def27}.}
we say $x \sqsubseteq^* y$ iff
\begin{enumerate}
\item $x\cap \omega_1 = y \cap \omega_1$,
\item $\ssup(x) = \ssup(y)$,
\item $\ssup(x\cap \gamma) = \ssup(y\cap \gamma)$ for any $\gamma\in x\cap E^\lambda_{\omega_1}$.
\end{enumerate}
\end{definition}

We will use the following equivalent formulation as the definition of $\ssr$ as proved in \cite[Lemma 2.2]{SakaiVelickovic}.

\begin{definition}
$\ssr$ asserts that for any $\lambda\geq \omega_2$ and any stationary $S\s[X]^{\aleph_0}$ that is closed upwards under $\sqsubseteq^*$, there exists $W\in[\lambda]^{\aleph_1}$ with $W\supseteq \omega_1$ such that $S\cap[W]^{\aleph_0}$ is stationary in $[W]^{\aleph_0}$.
\end{definition}

\begin{thm}\label{theoerm: SSR}
$\ssr$ implies that $\inds(\kappa,\omega_1)$ fails for every regular $\kappa\geq \omega_2$.
\end{thm}
\begin{proof}
The proof is similar to that of \cite[Theorem~2.1]{SakaiVelickovic}.
Let $\kappa\geq \omega_2$ be regular.
Suppose for the sake of contradiction that $\vec{C}=\langle C_{\alpha, i}\mid \alpha \in \Gamma, \ i(\alpha)\le i<\omega_1\rangle$
is an $\inds(\kappa,\omega_1)$-sequence.

We shall soon show that the following set is stationary in $[\kappa]^{\aleph_0}$. Define $S\subset [\kappa]^{\aleph_0}$ containing elements $x$ such that
\begin{enumerate}
\item $\sup (x)\not\in x$,
\item $x\cap \omega_1\in\omega_1$,
\item $i(\sup(x))\leq  x\cap \omega_1$,
\item $\exists\xi<\sup(x)$, $x \cap C_{\sup(x), x \cap \omega_1} \subset \xi \  \& \ \forall \beta\in C_{\sup(x), x\cap \omega_1}\backslash \xi, \,\cf(\min (x\backslash\beta))=\omega_1$.
\end{enumerate}

Let us note that by design, $S$ is closed upwards under $\sqsubseteq^*$. To see this, suppose that $x\in S$ and $y$ is such that $x\sqsubseteq^* y$. The only nontrivial point to check is (5). Let $D=C_{\sup (x), x\cap \omega_1}$.
First we check that $\sup (y\cap D) = \sup (x\cap D)$. If not, then we can take some $\gamma\in y\cap D\setminus(\sup (x\cap D)+1)$. By the fact that $x\in S$, we know that $cf(\gamma^*)=\omega_1$ where $\gamma^*=\min (x\backslash \gamma)$. As $\gamma\not\in x$, $\gamma^*>\gamma$. However, $\ssup(x\cap \gamma^*)\leq \gamma$ but $\ssup(y\cap \gamma^*)\geq \gamma+1$, contradicting the fact that $x\sqsubseteq^* y$. Let $\xi<\sup(x)$ bound $x\cap D$ and $y\cap D$. For any $\beta\in D\backslash (\xi+1)$, $\min y \backslash \beta =_{def} \beta_y \leq \beta_x = _{def}\min x \backslash \beta$. To see that $\beta_y = \beta_x$, suppose for the sake of contradiction that $\beta_y < \beta_x$. Then $\ssup(x\cap \beta_x)\leq \beta\leq \beta_y$ and $\ssup(y\cap \beta_x)\geq \beta_y+1$, contradicting the fact that $x\sqsubseteq^* y$.

Thus, if $\ssr$ were to hold, we could pick $W\in [\kappa]^{\aleph_1}$ with $W\supseteq \omega_1$ such that $S\cap [W]^{\aleph_0}$ is stationary in $[W]^{\aleph_0}$. Let $\gamma:=\sup(W)$.
There are two options here, each leading to a contradiction:
\begin{itemize}
\item[$\br$] If $\cf(\gamma)=\omega$, then $\gamma \in \Gamma$. Let $j := i(\gamma)$, and let $y$ be a countable cofinal subset of $C_{\gamma,j}$,
and note that $\{x\in [W]^{\aleph_0}\mid j\cup y\s x\}$ is a club in $[W]^{\aleph_0}$ disjoint from $S$, contradicting the fact that $S\cap [W]^{\aleph_0}$ is stationary.
\item[$\br$]If $\cf(\gamma) = \omega_1$, then let $\langle \beta_i \mid i < \omega_1 \rangle$
be an increasing enumeration of a club in $\gamma$. We can assume that $\cf(\beta_i) = \omega$ for
all $i < \omega_1$. Fix $\delta \in \Gamma \setminus \gamma$,
and define a function $g:\omega_1 \rightarrow \omega_1$ by letting, for all $i < \omega_1$,
$g(i)$ be the least $j < \omega_1$ such that $i(\delta) \leq j$ and $\beta_i \in \acc(C_{\delta,j})$.
Let $D := \{j \in \acc(\omega_1) \mid g[j] \subseteq j\}$, so $D$ is a club in $\omega_1$. Moreover,
for all $j \in D$, we have $\beta_i \in \acc(C_{\delta,j})$ for all $i < j$. It follows that
$\beta_j \in \acc(C_{\delta,j})$, so, by coherence, we have $i(\beta_j) \leq j$ and
$\beta_i \in \acc(C_{\beta_j,j})$ for all $i < j$. Let $E$ be the set of $x \in [W]^{\aleph_0}$
such that
\begin{itemize}
\item $x \cap \omega_1 \in D$;
\item $\sup(x) = \beta_{x \cap \omega_1}$; and
\item $x \cap C_{\sup(x), x \cap \omega_1}$ is unbounded in $\sup(x)$.
\end{itemize}
Then $E$ is disjoint from $S$, and, by the choice of $D$, $E$ is a club in $[W]^{\aleph_0}$,
contradicting the fact that $S \cap [W]^{\aleph_0}$ is stationary.
\end{itemize}

We now turn to showing that $S$ is indeed a stationary subset of $[\kappa]^{\aleph_0}$. To this end, let $f: [\kappa]^{<\omega}\to \kappa$ be given. Our goal is to find $x\in S$ closed under $f$. The proof is essentially the same as that in \cite{SakaiVelickovic}, so we just include a brief outline.

For each $j<\omega_1$, consider the following game $G_{f,j}:$ Players~I and II alternate choosing ordinals $<\lambda$, with Player~I starting the game. A run of a game takes the following form: at stage $n$, Player~I chooses $\alpha_n$, then Player~II chooses $\beta_n$, then Player I chooses $\gamma_n>\beta_n, \alpha_n$ of cofinality $\omega_1$. Player~I wins iff, letting $x:=\cl_f(\{\gamma_{n}\mid n<\omega\}\cup j)$, we have $x \cap \bigcup_{m\in \omega} [\alpha_m, \gamma_m) =\emptyset$ and $x\cap \omega_1=j$. Since this is an open game for Player~II, it is determined. An argument as in \cite[Lemma 2.3]{SakaiVelickovic} shows that for club many $j<\omega_1$, Player~I has a winning strategy $\sigma_{j}$ in the game $G_{f,j}$.
Fix a large enough $j<\omega_1$ such that
Player~I has a winning strategy in the game $G_{f,j}$
and such that, for stationarily many $\beta\in E^\kappa_\omega$, $i(\beta)\leq j$. Let $C\subset \kappa$ be a club subset that is closed under $f$ and the winning strategy of Player I $\sigma_j$.
Find some $M\prec H(\kappa^+)$ containing all relevant objects with $\beta :=\sup(M\cap\kappa)$ in $E^\kappa_\omega$ and $i(\beta)\leq j$. The rest of the proof is the same as \cite[Theorem 2.1, Claim 1]{SakaiVelickovic}, with $C_{\beta, j}$ playing the role of the ``$C_\delta$" in that proof.
\end{proof}

\begin{proof}[Proof of Theorem \ref{theorem: main2}(1)]
By Theorem~\ref{theoerm: SSR} and the fact that $\mm$ implies $\ssr$.
\end{proof}

\subsection{Another weakening}
We now show that $\mm$ is compatible with a different weakening of
$\square^{\ind}(\kappa, \omega_1)$, which we denote $\inde(\kappa, \theta)$. 
This will later be used to provide a sense in which
Theorem \ref{theorem: main1} is sharp. We begin with the definition of
this weakening.

\begin{definition} \label{def: weakening}
Let $\theta < \kappa$ be a pair of infinite regular cardinals.
The principle $\inde(\kappa, \theta)$
asserts the existence of a matrix
\[
\vec{C} = \langle C_{\alpha,i} \mid \alpha \in \acc(\kappa), ~
i(\alpha) \leq i < \theta \rangle
\]
satisfying the following requirements:
\begin{enumerate}
\item for all $\alpha \in \acc(\kappa)$, we have $i(\alpha) < \theta$, and
$\langle C_{\alpha,i} \mid i(\alpha) \leq i < \theta \rangle$ is a
$\s$-increasing sequence of clubs in $\alpha$, with $\acc(\alpha) =
\bigcup_{i(\alpha) \leq i < \theta} \acc(C_{\alpha,i})$;
\item for all $\alpha \in \acc(\kappa)$, $i(\alpha) \leq i < \theta$, and
$\bar{\alpha} \in \acc(C_{\alpha,i}) \cap E^\kappa_{\geq \theta}$, we have
$i(\bar{\alpha}) \leq i$ and $C_{\bar{\alpha},i} = C_{\alpha,i} \cap
\bar{\alpha}$;
\item for all $(\bar{\alpha}, \alpha) \in [\acc(\kappa)]^2$ and all
sufficiently large $i < \theta$, we have $C_{\bar{\alpha},i} =
C_{\alpha,i} \cap \bar{\alpha}$;
\item for every club $D$ in $\kappa$, there exists $\alpha \in \acc(D)
\cap E^\kappa_{\geq \theta}$ such that, for all $i < \theta$,
$D \cap \alpha \neq C_{\alpha,i}$.
\end{enumerate}
\end{definition}

Loosely speaking, the difference between $\inds(\kappa, \theta)$ and
$\inde(\kappa, \theta)$ is that, in a matrix $\vec{C}$ witnessing the latter,
if $(\bar{\alpha}, \alpha) \in [\acc(\kappa)]^2$ with $\cf(\bar{\alpha}) < \theta$,
then we do not require coherence of $C_{\bar{\alpha},i}$ and $C_{\alpha,i}$ for
\emph{all} $i < \theta$ such that $\bar{\alpha} \in \acc(C_{\alpha,i})$, but
only for all sufficiently large $i < \theta$. Note that $\inde(\kappa, \omega)$
is equivalent to $\square^{\mathrm{ind}}(\kappa, \omega)$. Hence, for notational
convenience, we will focus in this section on the case in which
$\theta > \omega$.

As should be expected of a square principle, $\inde(\kappa, \theta)$ is incompatible with the
weak compactness of $\kappa$.

\begin{definition}[\cite{paper34}]
A coloring $c: [\kappa]^2\to \theta$ witnesses $\U(\kappa,2,\theta,2)$ if for any $H\in [\kappa]^\kappa$, $c``[H]^2$ is cofinal in $\theta$.
\end{definition}

\begin{prop} \label{prop: inde_wc}
Suppose that $\theta < \kappa$ is a pair of infinite regular cardinals and $\inde(\kappa, \theta)$
holds.
Then there exists an $E^\kappa_{\ge\theta}$-closed subadditive witness to $\U(\kappa,2,\theta,2)$.

In particular, $\kappa$ is not weakly compact.
\end{prop}

\begin{proof}
Suppose $\vec{C} = \langle C_{\alpha,i} \mid \alpha \in \acc(\kappa), \ i(\alpha)
\leq i < \theta \rangle$ is a witness to $\inde(\kappa, \theta)$.
Using Clause~(3) of Definition~\ref{def: weakening}, we define a coloring
$c:[\kappa]^2 \rightarrow \theta$ via
$$c(\alpha,\beta):=\min\{j< \theta\mid j\ge\max\{i(\omega\cdot\alpha),i(\omega\cdot\beta)\}\ \&\ \forall i\in[j,\theta)\,[C_{\omega\cdot\alpha,i} = C_{\omega\cdot\beta,i} \cap \omega\cdot\alpha]\}.$$

\begin{claim} $c$ witnesses $\U(\kappa,2,\theta,2)$.
\end{claim}
\begin{proof} We need to show that for every $H\in[\kappa]^\kappa$, $c``[H]^2$ is cofinal in $\theta$.
Towards a contradiction, suppose $H\in[\kappa]^\kappa$ and $j<\theta$ are such that $c``[H]^2\s j$.
Then $D := \bigcup_{\alpha \in H} C_{\omega\cdot\alpha,j}$ is a club in $\kappa$.
Using Clause~(4) of Definition~\ref{def: weakening}, fix $\alpha\in\acc(D)\cap E^\kappa_{\ge\theta}$ such that, for all $i < \theta$,
$D \cap \alpha \neq C_{\alpha,i}$.
Set $\beta:=\min(H\setminus(\alpha+1))$.
Then $D\cap\alpha=C_{\omega\cdot\beta,j}\cap\alpha$. So $\alpha\in\acc(C_{\omega\cdot\beta,j})\cap E^\kappa_{\ge\theta}$,
and then Clause~(2) implies that $C_{\alpha,j}=C_{\omega\cdot\beta,j}\cap\alpha=D\cap\alpha$.
This is a contradiction.
\end{proof}

It thus immediately follows that $\kappa$ is not weakly compact.

\begin{claim} $c$ is $E^\kappa_{\ge\theta}$-closed.
\end{claim}
\begin{proof} Suppose that $\alpha<\beta<\kappa$ and $j<\theta$,
are such that $\sup\{\varepsilon<\alpha\mid c(\varepsilon,\beta)\le j\}=\alpha$;
we need to show that if $\alpha\in E^\kappa_{\ge\theta}$, then $c(\alpha,\beta)\le j$.

By our assumption, $\omega\cdot\varepsilon\in\acc(C_{\omega\cdot\beta,j})$ for cofinally many $\varepsilon<\alpha$,
and hence $\omega\cdot\alpha\in\acc(C_{\omega\cdot\beta,j})$. Thus, if $\alpha\in E^\kappa_{\ge\theta}$,
then $\omega\cdot\alpha\in\acc(C_{\omega\cdot\beta,j})\cap E^\kappa_{\ge\theta}$,
and then Clause~(2) of Definition~\ref{def: weakening} implies that $c(\alpha,\beta)\le j$.
\end{proof}

\begin{claim} $c$ is subadditive.
\end{claim}
\begin{proof} Let $\alpha<\beta<\gamma<\kappa$; we need to show that $c(\alpha,\gamma)\le\max\{c(\alpha,\beta),c(\beta,\gamma)\}$ and $c(\alpha,\beta)\le\max\{c(\alpha,\gamma),c(\beta,\gamma)\}$.

$\br$ Set $j:=\max\{c(\alpha,\beta),c(\beta,\gamma)\}$.
Then for every $i\in[j,\theta)$, $C_{\omega\cdot\alpha,i}=C_{\omega\cdot\beta,i}\cap\omega\cdot\alpha$ and $C_{\omega\cdot\beta,i}=C_{\omega\cdot\gamma,i}\cap\omega\cdot\beta$,
so that $C_{\omega\cdot\alpha,i}=C_{\omega\cdot\gamma,i}\cap\omega\cdot\alpha$. Consequently, $c(\alpha,\gamma)\le j$.

$\br$ Set $j:=\max\{c(\alpha,\gamma),c(\beta,\gamma)\}$.
Then for every $i\in[j,\theta)$, $C_{\alpha,i}=C_{\omega\cdot\gamma,i}\cap\omega\cdot\alpha$ and $C_{\beta,i}=C_{\omega\cdot\gamma,i}\cap\omega\cdot\beta$,
so that $C_{\omega\cdot\alpha,i}=C_{\omega\cdot\beta,i}\cap\omega\cdot\alpha$. Consequently, $c(\alpha,\beta)\le j$.
\end{proof}

This completes the proof.
\end{proof}

We now turn to proving that $\mm$ is compatible with $\inde(\kappa, \omega_1)$.
Hereafter, we roughly follow Section~7 of \cite{narrow_systems}. Fix
for now a pair of uncountable regular cardinals $\theta < \kappa$. We first
introduce a forcing to add a witness to $\inde(\kappa, \theta)$.

\begin{definition}\label{definition: ForcingIndexedMinus}
Define $\mathbb P^-(\kappa, \theta)$ to be the forcing poset consisting
of all conditions $p = \langle C^p_{\alpha, i} \mid \alpha\in\acc(\gamma^p + 1),
~ i(\alpha)^p \le i < \theta \rangle$ satisfying the following four requirements:
\begin{enumerate}
\item $\gamma^p \in \acc(\kappa)$;
\item for all $\alpha\in \acc(\gamma^p+1)$, we have $i(\alpha)^p < \theta$,
and $\langle C^p_{\alpha,i}\mid i(\alpha)^p \le i<\theta\rangle$ is a
$\s$-increasing sequence of clubs in $\alpha$, with
$\acc(\alpha) = \bigcup_{i(\alpha)\le i<\theta}\acc(C_{\alpha,i}^p)$;
\item for all $\alpha\in\Gamma^p$, $i(\alpha)^p \le i<\theta$, and
$\bar\alpha \in \acc(C^p_{\alpha,i}) \cap E^\kappa_{\geq \theta}$, we have
$i(\bar\alpha)^p \le i$ and $C_{\bar\alpha,i}^p=C_{\alpha,i}^p\cap\bar\alpha$;
\item for all $(\bar{\alpha}, \alpha) \in [\acc(\gamma^p+1)]^2$ and all
sufficiently large $i < \theta$, we have $C^p_{\bar{\alpha},i} =
C^p_{\alpha,i} \cap \bar{\alpha}$.
\end{enumerate}
$\bb{P}^-(\kappa, \theta)$ is ordered by end-extension.
\end{definition}

\begin{lemma}\label{lemma: directedclosed}
$\mathbb P^-(\kappa, \theta)$ is $\theta^+$-directed closed.
\end{lemma}
\begin{proof}
As $\mathbb P^-(\kappa,\theta)$ is tree-like, it suffices to verify that it is $\theta^+$-closed.
Suppose that we are given a strictly decreasing sequence $\vec p  = \langle p_\sigma \mid \sigma < \tau \rangle$ of conditions in $\mathbb P^-(\kappa,\theta)$,
with $\tau\in\acc(\theta^+)$.

Set $\gamma := \sup\{\gamma^{p_\sigma} \mid \sigma < \tau\}$.
We will define a lower bound $q$ for $\vec{p}$ with $\gamma^q = \gamma$.
For all $\alpha \in \acc(\gamma)$, let $\sigma < \tau$ be least
such that $\alpha \in \acc(\gamma^{p_\sigma})$, and set $i(\alpha)^q =
i(\alpha)^{p_\sigma}$ and, for all $i(\alpha)^q \leq i < \theta$, set
$C^q_{\alpha,i} = C^{p_\sigma}_{\alpha,i}$. To complete the definition of $q$,
it suffices to specify $i(\gamma)^q$ and $\langle C^q_{\gamma,i}
\mid i(\gamma)^q \leq i < \theta \rangle$.

Let $\nu := \cf(\tau) = \cf(\gamma)$, so $\nu \leq \theta$, and let
$\langle \beta_\eta \mid \eta < \nu \rangle$ be an increasing enumeration of a
club $D$ in $\gamma$ such that $\beta_\eta \in \acc(\gamma) \cap E^\tau_{<\nu}$
for all $\eta < \nu$ (such a club exists because $\cf(\gamma) \leq \theta$
and $\gamma$ is a limit of limit ordinals).
Suppose first that $\nu < \theta$. In this case, we can find a sufficiently large
$i^* < \theta$ such that, for all $(\eta,\xi) \in [\nu]^2$ and all
$i^* \leq i < \theta$, we have $C^q_{\beta_\eta, i} = C^q_{\beta_\xi,i} \cap
\beta_\xi$. Then set $i(\gamma)^q := i^*$ and, for all $i^* \leq i < \theta$,
set $C^q_{\gamma,i} := \bigcup_{\eta < \nu} C^q_{\beta_\eta,i}$. It is routine
to verify that $q$ thus defined is as desired.

Suppose now that $\nu = \theta$, and let $\langle i_\eta \mid
\eta < \theta \rangle$ be a continuous, strictly increasing sequence of ordinals
below $\theta$ such that $i_0 = 0$ and, for all $\xi < \xi' < \eta < \theta$ and all
$i_\eta \leq i < \theta$, we have $i(\beta_\xi)^q, i(\beta_{\xi'})^q<i_\eta$ and $C^q_{\beta_\xi, i} = C^1_{\beta_{\xi'},i}
\cap \beta_\xi$. Set $i(\gamma)^q := 0$. For all $i < \theta$, let $\eta <
\theta$ be such that $i_\eta \leq i < i_{\eta + 1}$, and set
$C^q_{\gamma,i} := D \cup \bigcup_{\xi < \eta} C^q_{\beta_\xi,i}$. Notice that
our choice of $i_\eta$ ensures that, for all $\xi < \eta$, we have
$C^q_{\beta_\xi,i} = C^q_{\gamma,i} \cap \beta_\xi$. It is again readily verified
that $q$ is as desired.
\end{proof}

\begin{lemma} \label{game}
$\mathbb P^-(\kappa, \theta)$ is $\kappa$-strategically closed.
\end{lemma}
\begin{proof}
We describe a winning strategy for Player II in $\Game_\kappa(\mathbb
P^-(\kappa, \theta))$.
Suppose $0 < \xi < \kappa$ is an even ordinal and $\langle p_\eta \mid \eta < \xi \rangle$ is a partial play of $\Game_\kappa(\mathbb P^-(\kappa, \theta))$.
Assume we have arranged inductively that, for all even nonzero ordinals $\eta < \xi$,
we have $\gamma^{p_\eta} < \gamma^{p_\xi}$, $i(\gamma^{p_\eta})^{p_\eta} = i(\gamma^{p_\xi}) = 0$, and, for all
$i < \theta$, $C^{p_\eta}_{\gamma^{p_\eta},i} = C^{p_\xi}_{\gamma^{p_\xi},i}
\cap \gamma^{p_\eta}$.

Suppose first that $\xi = \eta + 2$ for some even $\eta < \kappa$.
We shall define a condition  $p_\xi = \langle C^{p_\xi}_{\alpha, i}
\mid \alpha\in\acc(\gamma^{p_\xi} + 1), ~ i(\alpha)^{p_\xi} \le i < \theta
\rangle$ extending $p_{\eta+1}$.
First, set $\gamma^{p_\xi}:=\gamma^{p_{\eta+1}}+\omega$ and $i(\gamma^{p_\xi})^{p_\xi}
= 0$. To complete the definition of $p_\xi$, we only need to define
$$\langle C^{p_\xi}_{\gamma^{p_\xi}, i} \mid i < \theta \rangle.$$
First, fix $i^* < \theta$ such that, for all $i^* \leq i < \theta$, we have
either $\gamma^{p_\eta} = \gamma^{p_{\eta+1}}$ or $C^{p_\eta}_{\gamma^{p_\eta},i}
= C^{p_{\eta+1}}_{\gamma^{p_{\eta+1}},i} \cap \gamma^{p_\eta}$. Now, for
all $i < \theta$, define $C^{p_\xi}_{\gamma^{p_\xi},i}$ as follows.

$\br$ If $i < i^*$, then let
\[
C^{p_\xi}_{\gamma^{p_\xi},i} := C^{p_\eta}_{\gamma^{p_\eta},i} \cup
\{\gamma^{p_\eta}\} \cup \{\gamma^{p_{\eta + 1}} + n \mid n < \omega\}.
\]

$\br$ If $i \geq i^*$, then let
\[
C^{p_\xi}_{\gamma^{p_\xi},i} := C^{p_{\eta+1}}_{\gamma^{p_{\eta+1}},i}
\cup \{\gamma^{p_{\eta + 1}} + n \mid n < \omega\}.
\]

It is easily verified that $p_\xi$ forms a legitimate condition extending $p_{\eta + 1}$
satisfying the inductive hypothesis.

Next, suppose that $\xi$ is a limit ordinal. Let $\gamma^{p_\xi} := \sup_{\eta<\xi}\gamma^{p_\eta}$ and $i(\gamma^{p_\xi})^{p_\xi} = 0$. To complete the definition
of $p_\xi$, it remains to specify
$$\langle C^{p_\xi}_{\gamma^{p_\xi}, i} \mid i < \theta \rangle.$$
By our inductive hypothesis, we know that $\langle \gamma^{p_\eta} \mid \eta < \xi,
\eta \text{ is even}\rangle$ enumerates a club in $\gamma^{p_\xi}$ and, for all
$i < \theta$ and all even $\eta < \eta' < \xi$, we have $C^{p_\eta}_{\gamma^{p_\eta},i}
= C^{p_{\eta'}}_{\gamma^{p_{\eta'}},i} \cap \gamma^{p_\eta}$. Therefore, for each
$i < \theta$, we can set
\[
C^{p_\xi}_{\gamma^{p_\xi},i} := \bigcup \{C^{p_\eta}_{\gamma^{p_\eta},i} \mid
\eta < \xi, \eta \text{ is even}\}.
\]
It easy to see that $p_\xi$ is a lower bound for $\langle p_\eta \mid \eta < \xi \rangle$ and maintains the inductive hypothesis. This completes the description
of the winning strategy for Player II.
\end{proof}

So, forcing with $\mathbb P^-(\kappa, \theta)$ preserves all cardinalities and cofinalities $\le \kappa$. If, in addition, $\kappa^{< \kappa} = \kappa$, then $|\mathbb P^-(\kappa, \theta)| = \kappa$ and hence preserves all cardinalities and cofinalities.
The proof of Lemma \ref{game} makes it clear that, for every $\alpha < \kappa$, the
set $D_\alpha := \{p \in \mathbb P^-(\kappa, \theta) \mid \gamma^p\ge\alpha \}$ is dense in $\mathbb P^-(\kappa, \theta)$

\begin{lemma}
Let $G$ be $\mathbb P^-(\kappa, \theta)$-generic over $V$.
Set $\vec{{C}} := \bigcup G = \langle C_{\alpha, i} \mid \alpha \in \acc(\kappa),~ i(\alpha) \le i < \theta \rangle$.
Then $\vec{{C}}$ is an $\inde(\kappa,\theta)$-sequence.
\end{lemma}
\begin{proof} The only nontrivial thing to verify is that
$\vec{{C}}$ satisfies Clause~(4) of Definition~\ref{def: weakening}.
\begin{claim} \label{claim_3101}
Suppose that for every $i^*<\theta$, the set $\{\alpha\in E^\kappa_{\geq \theta} \mid i(\alpha)>i^*\}$ is stationary.
Then for every club $D$ in $\kappa$, there exists $\alpha\in\acc(D) \cap E^\kappa_{\geq \theta}$ such that, for all $i<\theta$, $D\cap\alpha\neq C_{\alpha,i}$.
\end{claim}
\begin{proof} Suppose that there is a club $D \s\kappa$ such that, for all $\alpha \in \acc(D)\cap E^\kappa_{\geq \theta}$,
there exists $i_\alpha < \theta$ for which $D \cap \alpha = C_{\alpha, i_\alpha}$.
Find a stationary set $S\s\acc(D)\cap E^\kappa_{\geq \theta}$ and some $i^*<\theta$ such that $i_\alpha=i^*$ for all $\alpha\in S$.
Then it easily follows that for every $\alpha \in \acc(D)\cap E^\kappa_{\geq \theta}$,
we have $D \cap \alpha = C_{\alpha, i^*}$.
\end{proof}

It thus suffices to verify that the hypothesis of Claim \ref{claim_3101} holds.
Work back in $V$.
Fix $p \in G$, a $\mathbb P^-(\kappa, \theta)$-name $\dot{D}$ such that $p \Vdash ``\dot{D}\text{ is club in }\kappa"$, and an ordinal $i^*<\theta$.
Build a strictly decreasing sequence $\vec{p} = \langle p_\eta \mid \eta < \theta \rangle$
of conditions in $\bb{P}^-(\kappa, \theta)$ below $p$ together with an increasing
sequence of ordinals $\langle \delta_\eta \mid \eta < \theta \rangle$ such that,
for all $\eta < \theta$, we have
\begin{itemize}
\item $\gamma^{p_\eta} < \delta_\eta < \gamma^{p_{\eta+1}}$; and
\item $p_{\eta+1} \Vdash \check{\delta}_\eta \in \dot{D}$.
\end{itemize}
Let $\gamma := \sup\{\delta_\eta \mid \eta < \theta\} = \sup\{\gamma^{p_\eta}
\mid \eta < \theta\}$. By Lemma \ref{lemma: directedclosed}, we can find a
lower bound $q_0$ for $\vec{p}$ such that $\gamma^{q_0} = \gamma$. The condition
constructed in the proof of that lemma satisfies $i(\gamma)^{q_0} = 0$. However,
if we alter $q_0$ to a condition $q$ simply by setting $i(\gamma)^q = i^* + 1$
and leaving $C^q_{\gamma,i} = C^{q_0}_{\gamma,i}$ for all $i^* \leq i < \theta$,
then $q$ is still a lower bound for $\vec{p}$. Moreover,
\[
q \Vdash ``\check{\gamma} \in \acc(D) \cap E^\kappa_{\geq \theta}
\text{ and } \dot{i(\alpha)} > i^*".
\]
By genericity, the conclusion follows.
\end{proof}
\begin{remark} The above forcing introduces a non-reflecting stationary subset of $E^\kappa_\theta$, e.g., the set of all $\alpha\in E^\kappa_\theta$ such that $i(\alpha)=0$ and $\otp(C_{\alpha,0})=\theta$.
\end{remark}

We now arrive at the proof of Theorem~\ref{theorem: main2}(2):

\begin{cor}\label{mmiscompatible} If $\mm$ holds, then for every regular uncountable cardinal $\kappa=\kappa^{<\kappa}$, in some cofinality-preserving forcing extension,
$\mm$ and $\inde(\kappa, \omega_1)$ both hold.
\end{cor}
\begin{proof} By \cite[Theorem~4.3]{MR1782117}, $\mm$ is preserved by any $\omega_2$-directed closed forcing.
\end{proof}

\begin{remark} Comparing Corollary~\ref{mmiscompatible} and Proposition~\ref{prop: inde_wc}
with Theorem~\ref{theoerm: SSR} and \cite[Theorem~4.4]{paper36},
we see that the pump-up feature of \cite[Corollary~5.20]{knaster_ii} is not available for the class of subadditive colorings.
\end{remark}

\subsection{Another interpolant} In \cite{hlh},
Hayut and Lambie-Hanson introduced the following definition as part of their investigation of
$\square(\kappa,\theta)$-sequences and stationary reflection principles.
\begin{definition}[{\cite[Definition~2.17]{hlh}}]\label{fulls} A $\square(\kappa,\theta)$-sequence $\langle\mathcal C_\alpha\mid\alpha<\kappa\rangle$
is said to be \emph{full} if the following set is cofinal in $\kappa$:
$$\Gamma:=\{\gamma<\kappa\mid \{\alpha<\kappa\mid \gamma\notin\bigcup\nolimits_{C\in\mathcal C_\alpha}\acc(C)\}\text{ is nonstationary in }\kappa\}.$$
\end{definition}
\begin{remark} $\square^{\ind}(\kappa, \theta)\implies\exists\text{ full }\square(\kappa,\theta)\text{-sequence}\implies\square(\kappa,\theta)$.
\end{remark}

Question~3 of \cite{hlh} asks whether $\square(\kappa,\theta)$ may always be witnessed by a full $\square(\kappa,\theta)$-sequence.
A negative answer follows from a result of Susice \cite{susice2019special} together with the following observation.
\begin{prop}\label{prop313} Suppose that there exists a full $\square(\kappa,\theta)$-sequence.
Then there exists a $\kappa$-Aronszajn tree $T$ with a $\theta$-ascending path.\footnote{See Definition~\ref{ascdef}.}
\end{prop}
\begin{proof}
Suppose $\vec{\mathcal C}=\langle\mathcal C_\alpha\mid\alpha<\kappa\rangle$ is a full $\square(\kappa,\theta)$-sequence.
For each $\alpha<\kappa$, let $\langle C_\alpha^i\mid i<\theta\rangle$ be some enumeration of $\mathcal C_\alpha$, with repetitions if necessary.
For each $i<\theta$, let $T^i$ be the tree $\mathcal T(\rho_0(\vec{C^i}))$ for the $C$-sequence $\vec{C^i}:=\langle C_\alpha^i\mid \alpha<\kappa\rangle$ (see \cite[\S 6.1]{todorcevicwalksbook}
for the definition of $\mathcal T(\rho_0(\vec{C^i}))$).
As each $\vec{C^i}$ is in particular a transversal for a $\square(\kappa,{<}\kappa)$-sequence,
$T^i$ is a $\kappa$-Aronszajn tree.\footnote{This is a standard argument. The proof that it is a $\kappa$-tree is similar to that of \cite[Claim~4.11.3]{paper34}.
The proof that it has no $\kappa$-branch is as that of the forward implication of \cite[Theorem~6.3.5]{todorcevicwalksbook}.}
Consequently, $T:=\bigcup_{i<\theta}T^i$ is a $\kappa$-Aronszajn tree.
A moment's reflection makes it clear that if $\vec{\mathcal C}$ is full, then $T$ admits a $\theta$-ascending path.
\end{proof}

In \cite{susice2019special}, Susice proved that $\square_{\omega_1,2}$ is consistent with the assertion that
all $\aleph_2$-Aronszajn trees are special. As $\square_{\omega_1,2}$ implies $\square(\omega_2,\omega)$,
it suffices to prove that if all $\aleph_2$-Aronszajn trees are special, then there are no full $\square(\omega_2,\omega)$-sequences.
But this follows from Laver's theorem that an $\omega_2$-Aronszajn tree with an $\omega$-ascending path is nonspecial (see \cite[Corollary~1.7]{lucke}).

It is worth noting that $\diamondsuit(\omega_1)$ holds in Susice's model, assuming it held in the ground model. The reason is that the forcing he used is countably closed and countably closed forcings are known to preserve $\diamondsuit(\omega_1)$.
It was proved in \cite{paper31} that $\diamondsuit(\omega_1)+\square_{\omega_1}$ gives an $\aleph_2$-Souslin tree,
and this model shows that $\square_{\omega_1}$ cannot be relaxed to $\square_{\omega_1,2}$.

\section{The impact of indecomposable ultrafilters}\label{section: ultrafiltersquare}

For the convenience of stating results, let us  define the following.

\begin{definition}\label{interval-indecomposable}
For $\theta < \kappa$, an ultrafilter $U$ is said to be \emph{$[\theta, \kappa)$-indecomposable} if it is $\mu$-indecomposable for all $\theta\in[\mu,\kappa)$.
\end{definition}

Note that this is equivalent to the assertion that, for every $\mu
< \kappa$ and every function $f:\kappa \rightarrow \mu$, there is $H \in [\mu]^{<\theta}$
such that $f^{-1}[H] \in U$.
Recall that an ultrafilter $U$ over a cardinal $\kappa > \aleph_1$ is
\emph{indecomposable} if it is uniform and $[\aleph_1,\kappa)$-indecomposable.

Note that an ultrafilter $U$ is $\aleph_0$-indecomposable if and only if it is $\aleph_1$-complete. Also,
if $U$ is a nonprincipal ultrafilter containing a set of cardinality $\mu$, then $U$ is
$\mu$-decomposable. We remark that, by a result of Kunen and Prikry \cite{kunen_prikry},
if $\mu$ is a regular
cardinal and $U$ is $\mu$-indecomposable, then it is also $\mu^+$-indecomposable. As a
result, if $\kappa$ carries a uniform indecomposable ultrafilter, then $\kappa$ cannot be the
successor of a regular cardinal.

\begin{fact}[Silver, {\cite[Lemma~2]{silver}}]\label{finest}
Suppose that $\theta$ is regular and $U$ is a uniform $[\theta,\kappa)$-indecomposable ultrafilter over a cardinal $\lambda$ with $\lambda\geq \kappa>2^{\theta}$
that is not $\theta$-complete.
Then there exist a $\mu<\theta$ and a map $\varphi: \lambda\rightarrow \mu$ that is a \emph{finest partition} associated to $U$. That is:
\begin{itemize}
\item for all $n<\mu$, $\varphi^{-1}[n]\notin U$;
\item for any $f: \lambda\to \gamma$ with $\gamma<\kappa$, there exists a function $g: \mu\to \gamma$ such that $f = g\circ \varphi \pmod U$.
\end{itemize}
\end{fact}

With a $[\theta, \kappa)$-indecomposable ultrafilter $U$ over $\kappa$ and a finest partition $\varphi: \kappa\rightarrow\mu$ associated with it, we can let $D:=\varphi^*(U)$ be the Rudin-Keisler projection of $U$ via $\varphi$. Then $D$ is a non-principal uniform ultrafilter on $\mu$ defined by putting $X \subseteq \mu$
in $D$ if and only if $\varphi^{-1}[X] \in U$.
The following theorem is due to Silver, whose proof is implicit in \cite{silver}. A countably complete version appeared as \cite[Theorem~7.5.26]{goldbergthesis}.
We include a proof of the following for completeness.

\begin{thm}[Silver]\label{theorem: silver}
Suppose $U$ is an ultrafilter satisfying the hypothesis of Fact \ref{finest}. Let $\varphi$ and $D$ be given as in the preceding discussion. The ultrapower embedding $j_U: V\to M_U$ can be factored as $k\circ j_D$ where $j_D: V\to M_D $ and $k: M_D\to M_U$ such that $k$ is $j_D(\eta)$-$M_D$-complete for all $\eta<\kappa$, namely, for any $\sigma\in M_D$ such that $M_D\models |\sigma|< j_D(\eta)$, we have $k(\sigma)=k``\sigma$.
\end{thm}

\begin{proof}
Recall that elements of $M_D$ and $M_U$ are of the form $[f]_D$ and $[g]_U$, where $f$ and $g$ are
functions with domains $\mu$ and $\lambda$, respectively.
Let $k: M_D\to M_U$ be defined by setting $k([f]_D)=[\bar{f}]_U$, where $\bar{f}=f\circ \varphi$. In particular, we have that $k\circ j_D = j_U$. To see that $k$ is elementary, for a formula $\psi(x_0, \ldots, x_{n-1})$ and $[f_i]_D\in M_D$ such that $M_D\models \psi([f_0]_D, \ldots, [f_{n-1}]_D)$, we know that $\{n\in \mu\mid V\models \psi(f_0(n), \ldots, f_{n-1}(n))\}\in D$. Since $D=\varphi^*(U)$, we know that $\{\alpha\in \lambda\mid V\models \psi(f_0(\varphi(\alpha)), \ldots, f_{n-1}(\varphi(\alpha)))\}\in U$, hence $M_U\models \psi([\bar{f}_0]_{U}, \ldots, [\bar{f}_{n-1}]_U)$.

It remains to check that $k$ is $j_D(\eta)$-$M_D$-complete for all $\eta<\kappa$. Fix $\eta<\kappa$. Let $X\in M_D$ be such that $M_D \models |X|< j_D(\eta)$. Let $f: \mu\to [V]^{\leq \eta}$ represent $X$ in $M_D$. In particular, $k(X)=[\bar{f}]_U$.
On the other hand, $k``X = \{k([g]_D)\mid g: \mu\rightarrow\mathcal V, \{n\in \mu\mid g(n)\in f(n)\}\in D\}$. It is clear that $k`` X\subset k(X)$. Let us check the other direction.
Let $[g]_U\in k(X)$, so we have $\{\alpha<\lambda\mid g(\alpha)\in \bar{f}(\alpha)\}\in U$. Since each $\bar{f}(\alpha)$ has size at most $\eta$, we can let $h: \lambda \to \eta$ be such that $g(\alpha)$ is the $h(\alpha)$-th element of $\bar{f}(\alpha)$. Here, for each $\alpha$, we fix some well ordering of $\bar{f}(\alpha)$ of order type $\leq \eta$. By the indecomposability assumption on $U$, we know that $h(\alpha)=_U r\circ \varphi$ for some $r: \mu\to \eta$. Define $g': \mu\to V$ such that $g'(n)$ is the $r(n)$-th element of $f(n)$.

We claim that $k([g']_D)=[g]_U$, which is clearly sufficient. Let $\bar{g}=g'\circ\varphi$ and consider $[\bar{g}]_U=k([g']_D)$. In short, we need to show $g=_U \bar{g}$. This amounts to showing that on a measure one set in $U$, $\bar{g}(\alpha)$ is  the $h(\alpha)$-th element of $\bar{f}(\alpha)$. To see this, note  that the following two sets belong to $U$:
\begin{itemize}
\item $A_0:=\{\alpha<\lambda\mid \bar f(\alpha)=f(\varphi(\alpha)), \bar g(\alpha)=g'(\varphi(\alpha)), h(\alpha)=r(\varphi(\alpha))\}$, and
\item $A_1:=\{\alpha<\lambda\mid g'(\varphi(\alpha)) \text{ is the }r(\varphi(\alpha))\text{-th element of }f(\varphi(\alpha))\}$,
\end{itemize}
so $A_0\cap A_1\in U$ is as desired.
\end{proof}

Let $W$ be a possibly external $M_D$-ultrafilter on $j_D(\kappa)$ derived from $k$ using $[\id]_U$.
In other words, for all $A \in M_D$ such that $M_D \models A \subseteq j_D(\kappa)$, we put
$A \in W$ if and only if $[\id]_U \in k(A)$. Then Theorem~\ref{theorem: silver} implies that $W$ is $M_D$-$j_D(\eta)$-complete for all $\eta<\kappa$. To see this, given $\mathcal{A}\subset W$ such that $\mathcal{A}\in M_D$ and $M_D\models |\mathcal{A}|< j_D(\eta)$, by Theorem~\ref{theorem: silver} it follows that $k(\mathcal{A})=k``\mathcal{A}$. In particular, $k(\bigcap \mathcal{A})=\bigcap k``\mathcal{A}$. Since for each $X\in \mathcal{A}$, $[\id]_U\in k(X)$, we have that $[\id]_U\in k(\bigcap \mathcal{A})$, namely, $\bigcap \mathcal{A}\in W$.

The following is due to Kunen and Goldberg \cite{goldberg}.
\begin{lemma}\label{lemma: approximation}
Let $U$ be as in the hypothesis of Fact \ref{finest}, and let $D, j_D, k$ be as in the conclusion of Theorem \ref{theorem: silver}. Furthermore, assume that $2^\gamma<\kappa$. Then for any $x\in [V]^{\gamma}$, $W\cap j_D(x)\in M_D$.
\end{lemma}
\begin{proof}
Let $\sigma=j_D(x)$. By Theorem \ref{theorem: silver}, $k(\sigma)=k``\sigma$. In $M_U$, let $B'=\{X\in k(\sigma)\mid [\id]_U\in X\}$. Since $M_D\models |\sigma|=j_D(\gamma)$, we can fix a bijection $\gamma: \sigma \leftrightarrow j_D(\gamma)$ in $M_D$ and let $B=k(\gamma)``B'$. Let $f: \lambda\to P(\gamma)$ be such that $j(f)([\id]_U)=B$. By the indecomposability assumption on $U$, there exists $g: \mu\to P(\gamma)$ such that $f=_U g\circ \varphi$. By the definition, $k([g]_D)=[g\circ \varphi]_U=[f]_U=B$. As a result, $(\gamma)^{-1}([g]_D)=W\cap j_D(x)$.
\end{proof}

\subsection{The $C$-sequence number}

In the remainder of this section, we investigate the effect of the existence of indecomposable ultrafilters
on other compactness phenomena, beginning with the $C$-sequence number.
The following result takes care of a case that is not covered by \cite[Lemma~4.12]{knaster_ii}.

\begin{thm} Suppose that $\theta,\kappa$ are infinite regular cardinals with $\kappa>2^\theta$.
If $\kappa$ carries a uniform $[\theta, \kappa)$-indecomposable ultrafilter,
then there exists a cardinal $\mu<\theta$ such that $\chi(\vec C)\le\mu$
for every transversal $\vec C$ for $\square(\kappa,{<}\kappa)$.
\end{thm}

\begin{proof}
Suppose $U$ is a $[\theta,\kappa)$-indecomposable ultrafilter on $\kappa$.
We may assume $U$ is $\theta$-incomplete for non-triviality.
By \cite[Theorem~2]{prikry}, we may also assume $U$ is weakly normal.
Let $\varphi: \kappa\rightarrow\mu$ with $\mu<\theta$ be given by Fact~\ref{finest}.
We shall prove that $\mu$ is as sought.

To this end, let $\vec C=\langle C_\beta\mid \beta<\kappa\rangle$ be some transversal for $\square(\kappa,{<}\kappa)$.
For each $\delta<\kappa$, define a function $f_\delta: \kappa\setminus\delta\rightarrow \mathcal P(\delta)$ via $$f_\delta(\beta):=C_\beta\cap \delta.$$
By the choice of $\vec C$, $|\im(f_\delta)|<\kappa$, so we may pick a map $g_\delta: \mu\to \mathcal P(\delta)$ satisfying that $f_\delta = g_\delta\circ \varphi\pmod U$.
Clearly, we can choose $g_\delta$ in a way that, for every $i<\mu$, there is some $\eta_{\delta,i}\geq \delta$ such that $g_\delta(i)=C_{\eta_{\delta,i}}\cap \delta$.
Set $D:=\varphi^*(U)$.
\begin{claim} Let $\gamma<\delta<\kappa$. Then
$g_\gamma\sq_D g_\delta$, i.e., $\{i<\mu\mid g_\gamma(i)\sq g_\delta(i)\}\in D$.
\end{claim}
\begin{proof}
This is because $B_\gamma:=\{\beta\in\kappa\setminus\gamma\mid f_\gamma(\beta)=g_\gamma(\varphi(\beta))\}$ and
$B_{\delta}:=\{\beta\in\kappa\setminus\delta\mid f_{\delta}(\beta)=g_{\delta}(\varphi(\beta))\}$ are both in $U$,
and for every $\beta\in B_\gamma\cap B_\delta$, $f_\gamma(\beta)\sq f_{\delta}(\beta)$.
But $D=\varphi^*(U)$, and hence $g_\gamma\sq_D g_\delta$.
\end{proof}
Consider the set $\Delta:=\{\delta\in E^\kappa_{>\mu}\mid \{i<\mu\mid \sup(g_\delta(i))=\delta\}\in D\}$.
\begin{claim} $\Delta$ covers a club relative to $E^\kappa_{>\mu}$.
\end{claim}
\begin{proof} Suppose not, so that $S:=E^\kappa_{>\mu}\setminus \Delta$ is stationary.
Define a regressive function $h:S\rightarrow \kappa$ via
$$h(\delta):=\min\{\varepsilon<\delta\mid \{i<\mu\mid \sup(g_\delta(i))<\varepsilon\}\in D\}.$$
The fact that $h$ is well-defined follows from the fact that $\cf(\delta)>\mu$ and $D$ is an ultrafilter on $\mu$.
Let $S'\s S$ be stationary on which $h$ is constant with value, say, $\varepsilon$.
Since $\gamma<\delta$ implies $g_\gamma\sq_D g_\delta$, we actually have $\{i<\mu\mid g_\delta(i)\s \varepsilon\}\in D$
for every $\delta<\kappa$.
By the weak normality of $U$, we can find some $\delta<\kappa$ for which the set
$$\{ \beta\in\acc(\kappa\setminus\varepsilon)\mid  \min(C_\beta\setminus\varepsilon+1))<\delta\}$$
is in $U$.
As a result, $\{i<\mu\mid g_{\delta}(i)\nsubseteq \varepsilon\}\in D$, which is a contradiction.
\end{proof}

As $\Delta$ is in particular an element of $[\kappa]^\kappa$,
it suffices to check that for every $\alpha<\kappa$, there is a set $b(\alpha)\in[\kappa]^\mu$ such that $\Delta\cap \alpha \s \bigcup_{\beta\in b(\alpha)} C_\beta$.
Set $\delta:=\min(\Delta\setminus(\alpha+1))$ and  $b(\alpha):=\{\eta_{\delta,i}\mid i<\mu\}$.
For each $\gamma\in \Delta\cap \alpha$, we know that $\{i<\mu\mid g_{\gamma}(i)\sq g_\delta(i)\}\in D$.
Recalling the definition of $\Delta$, it follows that $\{i<\mu\mid \gamma\in g_\delta(i)\}\in D$.
Altogether, $$\Delta\cap \alpha\s \bigcup_{i<\mu} g_\delta(i)=\bigcup_{i<\mu} C_{\eta_{\delta,i}}\cap\delta\s \bigcup_{\beta\in b(\alpha)}C_\beta,$$
as sought.
\end{proof}

\begin{cor}\label{theorem: $C$-sequenceNumber}
If a strongly inaccessible cardinal $\kappa$ carries a uniform $[\theta, \kappa)$-indecomposable ultrafilter where $\theta$ is regular, then $\chi(\kappa)<\theta$.\qed
\end{cor}
\begin{remark}
By Corollary~\ref{cor: indsquare} below, the preceding is optimal in the sense that we cannot strengthen the conclusion to $\chi(\kappa)\leq 1$.
\end{remark}

\begin{cor}[Prikry and Silver, \cite{prikry}]\label{cor: prikrysilver}
If a strongly inaccessible $\kappa$ carries a uniform indecomposable ultrafilter, then any finite collection of stationary subsets of $E^\kappa_{>\omega}$ reflects simultaneously.
\end{cor}

\begin{proof}
By \cite[Theorem~A(4)]{knaster_ii}, any finite collection of stationary subsets of $E^\kappa_{>\chi(\kappa)}$ reflects simultaneously.
Now appeal to Corollary~\ref{theorem: $C$-sequenceNumber}.
\end{proof}

\subsection{Ascent paths and narrow systems}

Let us recall Laver's definition of a $\mu$-ascent path and a couple of its generalizations.
\begin{definition}\label{ascdef}
Let $(T,<_T)$ be a tree of height $\kappa$, and let $\mu$ be an infinite cardinal.
\begin{itemize}
\item A \emph{$\mu$-ascent path} through $(T,<_T)$
is a sequence $\vec f =\langle f_\alpha\mid \alpha<\kappa\rangle$
satisfying the following two conditions:
\begin{enumerate}
\item for every $\alpha<\kappa$, $f_\alpha:\mu\rightarrow T_\alpha$;
\item for all $\alpha<\beta<\kappa$, $\{i<\mu\mid f_\alpha(i)<_T f_\beta(i)\}$ contains a tail in $\mu$.
\end{enumerate}
\item A \emph{$D$-ascent path} through $(T,<_T)$, where $D$ is a filter over $\mu$,
is a sequence $\vec f =\langle f_\alpha\mid \alpha<\kappa\rangle$ satisfying Clause~(1) above together with the following:
\begin{enumerate}
\item[(2')] for all $\alpha<\beta<\kappa$, $\{i<\mu\mid f_\alpha(i)<_T f_\beta(i)\}$ is in $D$.
\end{enumerate}
\item A \emph{$\mu$-ascending path} through $(T,<_T)$
is a sequence $\vec f =\langle f_\alpha\mid \alpha<\kappa\rangle$
satisfying Clause~(1) above together with the following:
\begin{enumerate}
\item[(2'')] for all $\alpha<\beta<\kappa$, there are $i,j<\mu$ such that $f_\alpha(i)<_T f_\beta(j)$.
\end{enumerate}
\end{itemize}
\end{definition}

We apply ideas similar to those of the previous subsection to show the following.

\begin{thm}\label{theoerem: ascent}
If a regular cardinal $\kappa>2^{\aleph_1}$ carries a uniform indecomposable ultrafilter, then every $\kappa$-Aronszajn tree
admits an $\omega$-ascent path.
\end{thm}

\begin{proof}
Let $U$ be the indecomposable ultrafilter.
We may assume $U$ is countably incomplete.
Let $\varphi: \kappa\to\omega$ be the finest partition associated with $U$ as given by Fact~\ref{finest}.
Consider $D:=\varphi^*(U)$, which is a nonprincipal ultrafilter on $\omega$.

Let $(T,<_T)$ be a given $\kappa$-Aronszajn tree.
Choose a transversal $\langle t_\beta\mid \beta<\kappa\rangle\in\prod_{\beta<\kappa}T_\beta$.
For each $\delta<\kappa$, define a map $f_\delta: \kappa\setminus\delta\to T_\delta$ via $$f_\delta(\beta):=t_\beta\restriction \delta.$$

Since $|T_\delta|<\kappa$, there is $g_\delta: \omega\to T_\delta$ such that $f_\delta = g_\delta \circ \varphi \pmod U$.
As before, for all $\gamma<\delta<\kappa$, $I_{\gamma,\delta}:=\{i<\omega\mid g_\gamma(i)<_T g_\delta(i)\}$ is in $D$.

For each $\gamma<\kappa$, define a map $h_\gamma: \kappa\setminus\gamma\rightarrow D$ via $h_\gamma(\delta):=I_{\gamma,\delta}$.
Since $U$ is indecomposable, we can find $\mathcal{I}_\gamma\in [D]^{\aleph_0}$ such that $\{\delta\in\kappa\setminus(\gamma+1)\mid I_{\gamma,\delta}\in \mathcal{I}_\gamma\}$ is in $U$.
Then, we find a pseudointersection $P_\gamma\in [\omega]^{\aleph_0}$ of the sets in $\mathcal{I}_\gamma$.
Finally, pick $P\in [\omega]^{\aleph_0}$ for which $\Gamma:=\{\gamma<\kappa\mid P_\gamma=P\}$ is cofinal in $\kappa$.

We check that for any pair $\gamma<\delta$ of ordinals from $\Gamma$, on a tail of $i\in P$, it is the case that $g_\gamma(i)<_T g_\delta(i)$.
Recalling that the following set is in $U$:
$$\{\eta\in\kappa\setminus(\gamma+1)\mid I_{\gamma,\eta}\in \mathcal{I}_\eta\}\cap \{\eta\in\kappa\setminus(\delta+1)\mid I_{\delta,\eta}\in \mathcal{I}_\eta\},$$
we may fix an $\eta<\kappa$ such that $I_{\gamma,\eta}\in \mathcal{I}_\gamma$ and $I_{\delta,\eta}\in \mathcal{I}_\delta$.
Consequently, $P=P_\gamma\s ^* A_{\gamma,\eta}$ and $P=P_\delta\s^* A_{\delta,\eta}$. Therefore, for co-finitely many $i\in P$, $g_\gamma(i)<_T g_\eta(i)$ and $g_\delta(i)<_T g_\eta(i)$, which implies $g_\gamma(i)<_T g_\delta(i)$.
It now easily follows that $(T,<_T)$ admits an $\omega$-ascent path.
\end{proof}

\begin{remark}\label{remark: D-ascent}
The above proof makes it clear that if $U$ is a uniform $[\theta, \kappa)$-indecomposable ultrafilter on $\kappa>2^\theta$ where $\theta$ is regular, then every $\kappa$-Aronszajn tree admits a $D$-ascent path,
where $D$ is a filter on some $\mu<\theta$.
To see this, if $U$ is $\theta$-incomplete, then we can apply Fact~\ref{finest} to get the finest partition $\varphi$ and let $D=\varphi^*(U)$. If $U$ is $\theta$-complete, then in fact $U$ is $\kappa$-complete. In this case, since there is a cofinal branch of the tree, $D$ can be taken to be a trivial filter on a singleton.
Note that by \cite[Lemmas 3.7 and 3.38(3)]{paper36},
if $\theta < \kappa$ are infinite regular cardinals and there exists a $\kappa$-Aronszajn tree with a $\theta$-ascent path,
then every uniform ultrafilter over $\kappa$ is $\theta$-decomposable.
\end{remark}

Given a binary relation $R$ on a set $X$,
for $a,b \in X$, we say that $a$ and $b$ are \emph{$R$-comparable} iff $a = b$, $a \mathrel{R} b$, or $b \mathrel{R} a$.
$R$ is \emph{tree-like} iff, for all $a,b,c \in X$, if $a \mathrel{R} c$ and $b \mathrel{R} c$, then $a$ and $b$ are $R$-comparable.

\begin{definition}[Magidor-Shelah, \cite{MgSh:324}]\label{system_def}
$\mathcal S = \langle \bigcup_{\alpha \in I}\{\alpha\} \times \theta_\alpha, \mathcal{R} \rangle$ is a \emph{$\kappa$-system} if all of the following hold:
\begin{enumerate}
\item $I \subseteq \kappa$ is unbounded and, for all $\alpha \in I$, $\theta_\alpha$ is a cardinal such that $0 < \theta_\alpha < \kappa$;
\item $\mathcal{R}$ is a set of binary, transitive, tree-like relations on $\bigcup_{\alpha \in I}\{\alpha\} \times \theta_\alpha$ and $0 < |\mathcal{R}| < \kappa$;
\item for all $R \in \mathcal{R}$, $\alpha_0, \alpha_1 \in I$, $\beta_0 < \theta_{\alpha_0}$, and $\beta_1 < \theta_{\alpha_1}$, if $(\alpha_0, \beta_0) \mathrel{R} (\alpha_1,\beta_1)$, then $\alpha_0 < \alpha_1$;
\item for every $(\alpha_0,\alpha_1)\in[I]^2$. there are $(\beta_0,\beta_1)\in \theta_{\alpha_0}\times \theta_{\alpha_1}$ and $R \in \mathcal{R}$ such that $(\alpha_0, \beta_0) \mathrel{R} (\alpha_1, \beta_1)$.
\end{enumerate}

Define $\mathrm{width}(\mathcal S) := \sup\{|\mathcal{R}|,\theta_\alpha \mid \alpha \in I\}$.
A $\kappa$-system $\mathcal S$ is \emph{narrow}  if $\mathrm{width}(\mathcal S)^+ < \kappa$.
For $R \in \mathcal{R}$, a \emph{branch of $\mathcal S$ through $R$} is a set $B\s \bigcup_{\alpha \in I}\{\alpha\} \times \theta_\alpha$ such that for all $a,b \in B$, $a$ and $b$ are $R$-comparable.
A branch $B$ is \emph{cofinal} iff $\sup\{ \alpha\in I\mid \exists \tau<\theta_\alpha~(\alpha,\tau)\in B\}=\kappa$.
\end{definition}

\begin{definition}[\cite{narrow_systems}]
The \emph{$(\theta, \kappa)$-narrow system property}, which is abbreviated $\nsp(\theta, \kappa)$, asserts that every narrow $\kappa$-system of width $<\theta$ has a cofinal branch.
\end{definition}

By \cite[Theorem~10.3]{narrow_systems},
$\mathsf{PFA}$ implies that $\nsp(\omega_1, \kappa)$ holds for all regular
$\kappa \geq \aleph_2$. (In fact, as the proof in \cite{narrow_systems} shows, $\mathsf{ISP}(\omega_2)$,
or, equivalently, $\mathsf{GMP}$, is enough to derive the desired conclusion.) Recall that, for a
regular cardinal $\kappa$, the \emph{tree property at $\kappa$}, denoted $\TP(\kappa)$, is the
assertion that there are no $\kappa$-Aronszajn trees.

\begin{thm}\label{theorem: NSP+TP}
Suppose that $\theta < \kappa$ are uncountable cardinals with $\kappa$ regular, $\nsp(\theta,   \kappa)$ holds,
and $\kappa$ carries a $[\theta, \kappa)$-indecomposable ultrafilter. Then
$\TP(\kappa)$ holds.
\end{thm}

Before giving the proof we note that if $2^\theta < \kappa$ and $\theta$ is regular, then we can just apply Remark~\ref{remark: D-ascent} to get the desired conclusion, since a $D$-ascent path through $T$, where $D$ is a uniform ultrafilter on $\mu<\theta$ and $T$ is a $\kappa$-tree, is clearly a $\kappa$-narrow system of width $<\theta$. However, as we demonstrate below, we do not need these extra assumptions.

\begin{proof} Let  $U$ be $[\theta, \kappa)$-indecomposable ultrafilter on $\kappa$.
Fix a $\kappa$-tree $(T,<_T)$ and we shall find a cofinal branch through it.
Choose a transversal $\langle t_\alpha\mid \alpha<\kappa\rangle\in\prod_{\alpha<\kappa}T_\alpha$.
For each $\alpha < \kappa$,   using the $[\theta, \kappa)$-indecomposability of $U$, fix
a set $S_\alpha \in [T_\alpha]^{<\theta}$ such that the following set is in $U$:
\[
X_\alpha := \{\beta \in[\alpha, \kappa )\mid t_\beta \restriction \alpha \in S_\alpha\}.
\]
We can then fix an unbounded set $I \subseteq \kappa$ and a cardinal
$\nu < \theta$ such that $|S_\alpha| = \nu$ for all $\alpha \in I$.

We claim that $\mathcal{S} = \langle \langle S_\alpha \mid \alpha \in I \rangle, \{<_T\} \rangle$
is a system of height $\kappa$ and width $\nu$. The only nontrivial thing to verify is the
requirement that, for every pair $(\alpha,\beta)\in[I]^2$, there are $s \in S_\alpha$ and
$t \in S_\beta$ such that $s <_T t$. To this end, fix such a pair $(\alpha,\beta)$ and then
fix $\gamma \in X_\alpha \cap X_\beta$. Then $t_\gamma \restriction \beta \in S_\beta$ and
$t_\gamma \restriction \alpha \in S_\alpha$, and clearly $t_\gamma \restriction \alpha
<_T t_\gamma \restriction \beta$, so we have found $s$ and $t$ as desired.

Now apply $\nsp(\theta, \kappa)$ to find a cofinal branch $b$ through $\mathcal{S}$.
Then $b \in \prod_{\alpha \in I'} S_\alpha$ for some cofinal $I' \subseteq I$ and, for every
$(\alpha,\beta)\in[I']^2$, we have $b(\alpha) <_T b(\beta)$. It follows that the $<_T$-downward closure
of $\{b(\alpha) \mid \alpha \in I'\}$ is a cofinal branch through $T$.
\end{proof}

\begin{cor} \label{pfa_indec_cor}
Suppose that $\mathsf{PFA}$ holds and $\kappa$ is a regular cardinal carrying a uniform
indecomposable ultrafilter. Then $\TP(\kappa)$ holds. In particular, if, in addition,
$\kappa$ is inaccessible, then it is in fact Ramsey.
\end{cor}

\begin{proof}
The ``in particular" part follows from Theorem \ref{theorem: NSP+TP} and a theorem of Ketonen \cite[Theorem 3.1]{Ketonen} stating that if a weakly compact cardinal carries a uniform indecomposable ultrafilter, then it is in fact Ramsey.
\end{proof}

We will improve this theorem in Section~\ref{section: FAandUltra}, showing that in fact, in such a situation, $\kappa$ must be measurable.

\subsection{The $\pr_1$ principle}
As explained in the introduction to \cite{paper18}, the following principle of Shelah is intimately connected with non-productivity of chain conditions.
Note that it becomes stronger as we increase the third and fourth parameters.

\begin{definition}[Shelah, \cite{shelah_productivity}]\label{def_pr1} Suppose $\theta,\chi\le\kappa$ are cardinals.
\begin{itemize}
\item
$\pr_1(\kappa, \kappa, \theta, \chi)$ asserts the existence of a coloring $c:[\kappa]^2 \rightarrow \theta$ such that,
for every  $\sigma<\chi$, for every  pairwise disjoint subfamily $\mathcal{B}  \subseteq [\kappa]^{\sigma}$ of size $\kappa$,
for every $\tau< \theta$, there are $a,b\in\mathcal B$ with $a<b$ such that $c[a\times b] = \{\tau\}$;
\item
$\pr_1(\kappa, \kappa, \theta, (2,\chi))$ asserts the existence of a coloring $c:[\kappa]^2 \rightarrow \theta$ such that,
for every $A\in[\kappa]^\kappa$,
for every  $\sigma<\chi$, for every  pairwise disjoint subfamily $\mathcal{B}  \subseteq [\kappa]^{\sigma}$ of size $\kappa$,
for every $\tau< \theta$, there are $\alpha\in A$ and $b\in\mathcal B$ with $\alpha<b$ such that $c[\{\alpha\}\times b] = \{\tau\}$.
\end{itemize}
\end{definition}

Clearly, $\pr_1(\kappa, \kappa, \theta, 1+\chi)$ implies $\pr_1(\kappa, \kappa, \theta, (2,\chi))$.
We now demonstrate a constraint on the fourth parameter when the source cardinal carries a uniform indecomposable ultrafilter.
The following generalizes a remark made at the end of Section~2 of \cite{paper18}.

\begin{thm} \label{theorem: pr}
Let $F$ be a uniform filter on $\mu$.
If $\kappa$ is a strongly inaccessible cardinal such that every $\kappa$-Aronszajn tree admits an $F$-ascent path, then $\pr_1(\kappa,\kappa, 2,\allowbreak(2,\mu^+))$ fails.
\end{thm}

\begin{proof}
Suppose for the sake of contradiction that $c: [\kappa]^2\rightarrow 2$ is a counterexample.
Since $\kappa$ is a strongly inaccessible, and $c$ in particular witnesses $\kappa\nrightarrow[\kappa]^2_2$,
the set $T:=\{c(\cdot, \beta)\restriction\alpha\mid\alpha\le\beta<\kappa\}$ forms a $\kappa$-Aronszajn tree,
so it must admit an $F$-ascent path.
This means that we can find $\langle \langle \beta_{\alpha,j}\mid  j< \mu\rangle\mid \alpha<\kappa\rangle$ such that:
\begin{itemize}
\item For all $\alpha<\kappa$, the set $b_\alpha:=\{\beta_{\alpha,j}\mid j< \mu\}$ is disjoint from $\alpha$;
\item for all $\alpha_0<\alpha_1<\kappa$,
for $F$-many $j<\mu$, $c(\cdot, \beta_{\alpha_0,j})\restriction \alpha_0 = c(\cdot, \beta_{\alpha_1, j})\restriction \alpha_0$.
\end{itemize}

Choose $D\in [\kappa]^\kappa$ such that for every $(\alpha,\beta)\in D$, $\sup(b_{\alpha})<\beta$.
For each $\alpha\in D$, if there are $\beta\in D\setminus(\alpha+1)$ and $i<2$ such that $c[\{\alpha\}\times b_\beta]=\{i\}$,
then in particular for all $\gamma\in D\setminus(\beta+1)$, for $F$-many $j<\mu$, $c(\alpha, \beta_{\gamma,j})=i$. We call such $\gamma$ \emph{good for $\alpha$}.
Next we find $E\in [D]^\kappa$ and $i<2$ such that for every $\alpha\in E$, either no $\beta\in E\setminus(\alpha+1)$ is good for $\alpha$ or every $\beta\in E\setminus(\alpha+1)$ is good for $\alpha$ as witnessed by $i$.
Finally, as $c$ is a witnesses $\pr_1(\kappa,\kappa, 2, (2,\mu^+))$, we can find $(\alpha,\beta)\in[E]^2$ such that $c[\{\alpha\}\times b_\beta]=\{1-i\}$. But then this contradicts the fact that if $\beta$ is good for $\alpha$ then $i$ must be the witnessing color.
\end{proof}

\begin{cor}
Suppose that $\kappa$ is a strongly inaccessible cardinal, $\theta \in\reg(\kappa)$, and $\kappa$ carries a
uniform $[\theta, \kappa)$-indecomposable ultrafilter. Then $\pr_1(\kappa, \kappa, 2, \theta)$
fails.
\end{cor}

\begin{proof}
This follows from Theorem~\ref{theorem: pr} and Remark~\ref{remark: D-ascent}.
\end{proof}

\subsection{Indexed square} \label{subsection: ind_square_ind_uf}

In this section, we demonstrate that $\square^{\mathrm{ind}}(\kappa, \theta)$ is compatible
with the existence of a uniform $[\theta^+, \kappa)$-ultrafilter on $\kappa$. By the following
fact, this is sharp.

\begin{fact}[{\cite[Theorem~4.4 and Lemma~3.38(3)]{paper36}}]\label{fact: indsVSindecom}
Suppose that $\theta < \kappa$ is a pair of infinite regular cardinals such that $\inds(\kappa, \theta)$
holds. Then every uniform ultrafilter over $\kappa$ is $\theta$-decomposable.
\end{fact}

Fix for now a pair of infinite regular cardinals $\theta < \kappa$, and let $\mathbb{P} = \mathbb{P}(\kappa, \theta)$
be the forcing to add a $\square^{\ind}(\kappa, \theta)$-sequence introduced in \cite[\S 7]{narrow_systems}.
Conditions in $\mathbb{P}$ are thus all sequences $p = \langle C^p_{\alpha,i} \mid \alpha
\in \acc(\gamma^p + 1), ~ i(\alpha)^p \leq i < \theta \rangle$ such that
\begin{itemize}
\item $\gamma^p \in \acc(\kappa)$;
\item for all $\alpha \in \acc(\gamma^p + 1)$, we have $i(\alpha)^p < \theta$ and
$\langle C^p_{\alpha,i} \mid i(\alpha)^p \leq i < \theta \rangle$ is a $\subseteq$-increasing
sequence of clubs in $\alpha$, with $\acc(\alpha) = \bigcup_{i(\alpha) \leq i < \theta}
\acc(C^p_{\alpha,i})$;
\item for all $\alpha \in \acc(\gamma^p + 1)$, $i(\alpha)^p \leq i < \theta$, and
$\bar{\alpha} \in \acc(C^p_{\alpha,i})$, we have $i(\bar{\alpha})^p \leq i$ and
$C^p_{\bar{\alpha},i} = C^p_{\alpha,i} \cap \bar{\alpha}$.
\end{itemize}
$\mathbb{P}$ is ordered by end-extension.

Let $\dot{\vec{C}} = \langle \dot{C}_{\alpha, i} \mid \alpha < \kappa, ~ \dot{i(\alpha)} \leq i < \theta
\rangle$ be the canonical $\mathbb{P}$-name for the generically-added $\square^{\ind}(\kappa,
\theta)$-sequence. For each $i < \theta$, let $\dot{\mathbb{T}}_i$ be a $\mathbb{P}$-name for the
poset to thread the $i^{\mathrm{th}}$ column of $\dot{\vec{C}}$. More precisely,
the conditions of $\dot{\mathbb{T}}_i$ are forced to be the elements of $\{\dot{C}_{\alpha, i} \mid
\alpha < \lambda \wedge \dot{i(\alpha)} \leq i\}$, and the ordering is end-extension.

\begin{fact}[{\cite[Lemma 3.18]{hlh}}] \label{dense_closed_lemma}
\begin{enumerate}
\item For all $i < \theta$, the two-step iteration $\mathbb{P} \ast \dot{\mathbb{T}}_i$ has a
dense $\kappa$-directed closed subset.
\item In $V^{\mathbb{P}}$, there is a system of commuting projections
$\langle \pi_{ij} : \mathbb{T}_i \rightarrow \mathbb{T}_j \mid i \leq j < \theta \rangle$ defined
by letting $\pi_{ij}(C_{\alpha,i}) = C_{\alpha, j}$ for all $i \leq j < \theta$ and
$C_{\alpha,i} \in \mathbb{T}_i$.
\end{enumerate}
\end{fact}

The dense subset of $\mathbb{P} \ast \dot{\mathbb{T}}_i$ referenced in Clause (1) of the above fact can
be taken to be the collection of all $(p,\dot{t})$ such that $p \Vdash_{\mathbb{P}} \dot{t} =
C^p_{\gamma^p,i}$. We will refer to the set of such pairs as $\mathbb{U}_i$. It follows, that, if
$\kappa^{<\kappa} = \kappa$, then $\mathbb{P} \ast \dot{\mathbb{T}}_i$ is forcing equivalent to
$\mathrm{Add}(\kappa, 1)$, the forcing to add a single Cohen subset to $\kappa$.

\begin{lemma} \label{generic_thread}
Suppose that $i < j < \theta$. Then
\[
\Vdash_{\mathbb{P} \ast \dot{\mathbb{T}}_j} ``\text{for every club } D \subseteq \kappa, \text{ there
is } \alpha \in \acc(D) \text{ such that } \dot{i(\alpha)} > i".
\]
In particular, forcing with $\mathbb{T}_j$ over $V^{\mathbb{P}}$ does not add a thread through the
$i^{\mathrm{th}}$ column of the generic $\square^{\ind}(\kappa, \theta)$-sequence.
\end{lemma}

\begin{proof}
Suppose that $(p_0,\dot{t}_0) \in \mathbb{P} \ast \dot{\mathbb{T}_j}$ and
$\dot{D}$ is a name for a club in $\kappa$. We can assume that $(p_0, \dot{t}_0) \in
\mathbb{U}_j$. Recursively define a decreasing sequence $\langle (p_n, \dot{t}_n) \mid n < \omega
\rangle$ from $\mathbb{U}_j$ together with an increasing sequence of ordinals
$\langle \alpha_n \mid n < \omega \rangle$ such that, for all $n < \omega$, we have
\begin{itemize}
\item $\gamma^{p_n} < \alpha_n < \gamma^{p_{n+1}}$; and
\item $(p_{n+1}, \dot{t}_{n+1}) \Vdash \alpha_n \in \dot{D}$.
\end{itemize}
The construction is straightforward. At the end, let $\gamma := \sup\{\gamma^{p_n} \mid n < \omega\}
= \sup\{\alpha_n \mid n < \omega\}$. For $j \leq k < \theta$ and $m < n < \omega$, note
that $C^{p_m}_{\gamma^{p_m}, k} = C^{p_n}_{\gamma^{p_n}, k} \cap \gamma^{p_m}$; for such $k$, let
$E_k := \bigcup \{C^{p_n}_{\gamma^{p_n},k} \mid n < \omega\}$.
Define a condition $p$ extending each $p_n$ by setting
$\gamma^p := \gamma$, $i(\gamma)^p := j$, and, for all $k \in [j, \theta)$,
$C^p_{\gamma,k} := E_k$. Let $\dot{t}$ be a $\mathbb{P}$-name for $E_j$. Then
\begin{itemize}
\item $(p,\dot{t}) \in \mathbb{U}_j$ is a lower bound for $\langle (p_n, \dot{t}_n) \mid
n < \omega \rangle$;
\item $(p,\dot{t}) \Vdash \gamma \in \acc{\dot{D}}$;
\item $(p,\dot{t}) \Vdash \dot{i(\gamma)} = j > i$.
\end{itemize}
Since $(p_0, \dot{t}_0)$ and $\dot{D}$ were chosen arbitrarily, this completes the proof.
\end{proof}

Let $G$ be $\mathbb{P}$-generic over $V$, and let $\vec{C} = \langle C_{\alpha,i} \mid
\alpha < \kappa, ~ i(\alpha) \leq i < \theta \rangle$ be $\bigcup G$. Temporarily move to $V[G]$.
For each $i < \theta$, let $\bb{T}_i$ be the interpretation of $\dot{\bb{T}}_i$. Note that
forcing with $\bb{T}_i$ over $V[G]$ adds a thread through the $i^{\mathrm{th}}$ column of
$\vec{C}$, i.e., a club $D \subseteq \kappa$ such that, for all
$\alpha \in \acc(\kappa)$, we have $D \cap \alpha = C_{\alpha,i}$.

\begin{prop} \label{compatible_prop}
Suppose that $i_0 \leq i_1 < \theta$, $t_0 \in \mathbb{T}_{i_0}$, and $t_1 \in \mathbb{T}_{i_1}$.
Then, for all sufficiently large $j < \theta$, the conditions $\pi_{i_0j}(t_0)$ and
$\pi_{i_1j}(t_1)$ are compatible in $\mathbb{T}_j$.
\end{prop}

\begin{proof}
Let $\alpha_0, \alpha_1 \in \acc(\kappa)$ be such that $t_0 = C_{\alpha_0, i_0}$ and
$t_1 = C_{\alpha_1, i_1}$. If $\alpha_0 = \alpha_1$, then $\pi_{i_0j}(t_0) = \pi_{i_1j}(t_1)$
for all $i_1 \leq j < \theta$. If $\alpha_0 < \alpha_1$, then, since $\vec{C}$ is a
$\square^{\ind}(\kappa, \theta)$-sequence, there is $j_0 \geq i_1$ such that
$\alpha_0 \in \acc(C_{\alpha_1,j_0})$. Then, for all $j \geq j_0$, we have
$\pi_{i_0j}(t_0) \leq_{\mathbb{T}_j} \pi_{i_1j}(t_1)$. The case in which
$\alpha_1 < \alpha_0$ is symmetric.
\end{proof}

\begin{thm} \label{ind_sq_thm}
Suppose that $\theta < \kappa$ are regular, $\kappa$ is measurable,
$\mathbb{P} = \mathbb{P}(\kappa, \theta)$, and the
measurability of $\kappa$ is indestructible under forcing with
$\mathrm{Add}(\kappa, 1)$. Suppose also that $W$ is a uniform ultrafilter
over $\theta$. Then, in $V^{\mathbb{P}}$, there is a uniform ultrafilter
$U$ over $\kappa$ such that, for all $\mu < \kappa$,
\[
(U \text{ is }\mu\text{-decomposable}) \iff (W \text{ is }\mu\text{-decomposable}).
\]
In particular, $U$ is $[\theta^+,\kappa)$-indecomposable.
\end{thm}

\begin{proof}
Let $\dot{\vec{C}} = \langle \dot{C}_{\alpha,i} \mid \alpha < \kappa,
\dot{i(\alpha)} \leq i < \theta \rangle$ be the canonical $\mathbb{P}$-name for the
generic $\square^{\ind}(\kappa, \theta)$-sequence. For each $i < \theta$,
let $\dot{\mathbb{T}}_i$ be a $\mathbb{P}$-name for the forcing to add a thread
through the $i^{\mathrm{th}}$ column of $\dot{\vec{C}}$. By
Fact~\ref{dense_closed_lemma}, for each $i < \theta$, $\mathbb{P} \ast
\dot{\mathbb{T}}_i$ is forcing equivalent to $\mathrm{Add}(\kappa, 1)$.
Therefore, by our assumption about the indestructibility of the measurability
of $\kappa$, we can fix a $\mathbb{P} \ast \dot{\mathbb{T}}_i$-name
$\dot{U}_i$ for a normal $\kappa$-complete ultrafilter over $\kappa$.

\begin{claim} \label{concentration_claim}
Let $i < \theta$. Then
\[
\Vdash_{\mathbb{P} \ast \dot{\mathbb{T}}_i} \{\alpha < \kappa \mid
\dot{i(\alpha)} = i\} \in \dot{U}_i.
\]
\end{claim}

\begin{proof}
By the fact that $\dot{U}_i$ is forced to be $\kappa$-complete, there is
a name $\dot{j}$ for an ordinal below $\theta$ such that $\Vdash_{\mathbb{P} \ast \dot{\mathbb{T}}_i}\{\alpha < \kappa \mid
\dot{i(\alpha)} = \dot{j}\} \in \dot{U}_i$. Note first that, in $V^{\mathbb{P} \ast
\dot{\mathbb{T}}_i}$, there is a club $D \subseteq \kappa$ through the
set $\{\alpha < \kappa \mid i(\alpha) \leq i\}$. As a result, by the normality of
$U_i$, $\dot{j}$ is forced to be at most $i$. On the other hand, if
$j < i$ in $V^{\mathbb{P} \ast \dot{\mathbb{T}}_i}$, then, letting
$k : V^{\mathbb{P} \ast \dot{\mathbb{T}}_i} \rightarrow M$ be the ultrapower
map with respect to $U_i$, we can conclude that $k(\vec{C})_{\kappa,j}$ is
defined and is a thread through the $j^{\mathrm{th}}$ column of
$\vec{C}$, contradicting Lemma~\ref{generic_thread}.
\end{proof}

Let $G$ be $\mathbb{P}$-generic over $V$ and let $H_0$ be $\mathbb{T}_0$-generic
over $V[G]$. For each $i < \theta$, the projection $\pi_{0,i}$ induces a filter
$H_i$ that is $\mathbb{T}_i$-generic over $V[G]$. Let $U_i$ denote the
realization of $\dot{U}_i$ in $V[G * H_i]$. Note that $U_i \in V[G * H_0]$ for
all $i < \theta$. Only $U_0$ is an ultrafilter in $V[G*H_0]$, but each
$U_i$ is a normal ultrafilter with respect to sets (and sequences of sets)
in $V[G]$.

In $V[G*H_0]$, define an ultrafilter $U$ on $\mathcal{P}(\kappa)^{V[G]}$ as follows.
For all $X \in \mathcal{P}(\kappa)^{V[G]}$, put $X \in U$ if and only if
$\{i < \theta \mid X \in U_i\} \in W$. Note that $\mathcal{P}(\theta)^{V[G *
H_0]} = \mathcal{P}(\theta)^V$, so $W$ remains an ultrafilter in
$V[G * H_0]$. It follows that $U$ is in fact an ultrafilter on
$\mathcal{P}(\kappa)^{V[G]}$. We defined $U$ in $V[G * H_0]$, but we now show that
it is in fact in $V[G]$. Work for now in $V[G]$, and let $\dot{U}$ be a
$\mathbb{T}_0$-name for $U$.

\begin{claim}
For every $X \in \mathcal{P}(\kappa)$, either $\Vdash_{\mathbb{T}_0} X
\in \dot{U}$ or $\Vdash_{\mathbb{T}_0} X \notin \dot{U}$.
\end{claim}

\begin{proof}
Suppose for sake of contradiction that $X \subseteq \kappa$ and there
are $t, t' \in \mathbb{T}_0$ such that $t \Vdash X \in \dot{U}$ and
$t' \Vdash X \notin \dot{U}$. By extending $t$ and $t'$ if necessary,
we can fix sets $Y, Y' \in W$ such that
\begin{itemize}
\item for all $i \in Y$, $t \Vdash_{\mathbb{T}_0} X \in \dot{U}_i$; and
\item for all $i \in Y'$, $t' \Vdash_{\mathbb{T}_0} X \notin \dot{U}_i$.
\end{itemize}
Since, for each $i < \theta$, $\pi_{0i}:\mathbb{T}_0 \rightarrow
\mathbb{T}_i$ is a projection, and since $\dot{U}_i$ is a $\mathbb{T}_i$-name,
this implies that
\begin{itemize}
\item for all $i \in Y$, $\pi_{0i}(t) \Vdash_{\mathbb{T}_i} X \in \dot{U}_i$; and
\item for all $i \in Y'$, $\pi_{0i}(t') \Vdash_{\mathbb{T}_i} X \notin \dot{U}_i$.
\end{itemize}
By Proposition~\ref{compatible_prop}, we can find $j \in Y \cap Y'$ such that
$\pi_{0j}(t)$ and $\pi_{0j}(t')$ are compatible in $\mathbb{T}_j$. But this
leads to a contradiction, since the two conditions decide the statement
``$X \in \dot{U}_j$" in opposite ways.
\end{proof}

It follows that $U$ is in fact definable in $V[G]$. Since each $\dot{U}_i$ is
forced to be a uniform ultrafilter, it follows that $U$ is uniform.
Also, since each $\dot{U}_i$ is forced to be normal, and in particular to
concentrate on the set of limit ordinals below $\kappa$, $U$ also concentrates
on the set of limit ordinals below $\kappa$.
It remains to check that $U$ has the desired spectrum of decomposability.
To this end, fix an infinite cardinal $\mu < \kappa$.

Suppose first that $W$ is $\mu$-decomposable, as witnessed by a function
$g : \theta \rightarrow \mu$. Define a function $f : \kappa \rightarrow \mu$
by setting $f(\alpha) = g(i(\alpha))$ for all $\alpha \in \acc(\kappa)$
(recall that $i(\alpha)$ is the least ordinal $i$ such that $C_{\alpha,i}$
is defined). We claim that $f$ witnesses that $U$ is $\mu$-decomposable.
Suppose for sake of contradiction that $H \in [\mu]^{<\mu}$ is such that
$f^{-1}[H] \in U$. Move to $V[G * H_0]$. By definition of $U$ and $f$, we have
\[
\{i < \theta \mid \{\alpha \in \acc(\kappa) \mid g(i(\alpha)) \in H\} \in U_i\}
\in W.
\]
By Claim~\ref{concentration_claim}, each $U_i$ concentrates on the set
$\{\alpha \in \acc(\kappa) \mid i(\alpha) = i\}$, so the above expression
simplifies to
\[
\{i < \theta \mid g(i) \in H\} \in W,
\]
i.e., $g^{-1}[H] \in W$, contradicting the fact that $g$ witnesses the
$\mu$-decomposability of $W$.

Suppose next that $W$ is $\mu$-indecomposable; we must show that $U$ is also
$\mu$-indecomposable. To this end, fix a function $f:\kappa \rightarrow \mu$.
Move to $V[G \ast H_0]$. Using the $\kappa$-completeness of each $U_i$,
define a function $g:\theta \rightarrow \mu$
by letting $g(i)$ be the unique $\eta < \mu$ such that $f^{-1}\{\eta\} \in
U_i$ for all $i < \theta$. Since $W$ is $\mu$-indecomposable, we can find
$H \in [\mu]^{<\mu}$ such that $Y := g^{-1}[H] \in W$. Note that $g$ and $H$ are in
$V$, since $\mathbb{P} \ast \dot{\mathbb{T}}_0$ is $\kappa$-distributive.
Now, for all $i \in Y$, we have $f^{-1}[H] \supseteq f^{-1}\{g(i)\} \in U_i$.
Since $Y \in W$, it follows that $f^{-1}[H] \in U$. Since $f$ was arbitrary,
it follows that $U$ is $\mu$-indecomposable.
\end{proof}

Theorem~\ref{theorem: main3} follows from Theorem~\ref{ind_sq_thm}.
One application of this result is the construction of a model in which an
inaccessible cardinal $\kappa$ carries a uniform indecomposable ultrafilter and only
satisfies the minimal amount of stationary reflection implied by the existence of
such an ultrafilter. The following corollary shows
that Corollary~\ref{cor: prikrysilver} is consistently sharp in two ways.

\begin{cor}\label{cor: indsquare}
Suppose that $\kappa$ is a measurable cardinal. Then there is a forcing
extension in which the following all hold:
\begin{enumerate}
\item $\kappa$ is strongly inaccessible;
\item $\kappa$ carries a uniform indecomposable ultrafilter;
\item there is a non-reflecting stationary subset of $E^\kappa_\omega$;
\item for every stationary subset $S \subseteq \kappa$, there is a family of
countably many stationary subsets of $S$ that does not reflect simultaneously.
\end{enumerate}
\end{cor}

\begin{proof}
We can assume that the measurability of $\kappa$ is indestructible under
$\mathrm{Add}(\kappa, 1)$. Let $\mathbb{P} = \mathbb{P}(\kappa, \omega)$.
$V^{\mathbb{P}}$ is the desired model. Since $\mathbb{P}$ is $\kappa$-distributive,
$\kappa$ remains strongly inaccessible there, and, by
Theorem \ref{ind_sq_thm}, $\kappa$ carries a uniform indecomposable ultrafilter. The existence
of a non-reflecting stationary subset of $E^\kappa_\omega$ follows from
\cite[Theorem 3.4(5)]{knaster_ii}, and Clause (4) in the statement of the
theorem follows from \cite[Theorem 2.18]{hlh} and the observation that a
$\square^{\ind}(\kappa, \omega)$-sequence is a full
$\square(\kappa, {<} \omega_1)$-sequence in the sense of
Definition~\ref{fulls}.
\end{proof}

Note that, in the setup for Theorem~\ref{ind_sq_thm}, if $\theta$ is measurable,
then by letting $W$ be a $\theta$-complete ultrafilter over $\theta$, we can
require that the uniform ultrafilter $U$ we obtain over $\kappa$ in the forcing
extension is $\theta$-complete. With a bit more care, we can produce some
variations on results of Gitik from \cite{gitik}.
Recall that a cardinal
$\kappa$ is \emph{$\theta$-strongly compact} if every $\kappa$-complete filter
over a set $A$ can be extended to a $\theta$-complete ultrafilter over $A$.

\begin{thm}
Suppose that $\theta < \kappa$ are cardinals such that $\theta$ is measurable,
$\kappa$ is $\theta$-strongly compact, and the $\theta$-strong compactness of
$\kappa$ is indestructible under forcing with $\mathrm{Add}(\kappa,1)$. Then there is a
cofinality-preserving forcing extension in which
$\square^{\ind}(\kappa, \theta)$ holds and $\kappa$ is
$\theta$-strongly compact.
\end{thm}

\begin{proof}
Let $\mathbb{P} = \mathbb{P}(\kappa, \theta)$ and, for $i < \theta$, let
$\dot{\mathbb{T}}_i$ be a $\mathbb{P}$-name for the forcing to add a
thread through the $i^{\mathrm{th}}$ column of the generically added
$\square^{\ind}(\kappa, \theta)$-sequence. Let $W$ be a normal measure
over $\theta$. Let $G$ be $\mathbb{P}$-generic over $V$, and move to
$V[G]$, which is our desired model. Let $A$ be a set, and let
$F$ be a $\kappa$-complete filter over $A$. For each $i < \theta$, let
$\dot{F}_i$ be a $\mathbb{T}_i$-name for the filter over $A$ generated by
$F$ in $V[G]^{\mathbb{T}_i}$. Because $\mathbb{T}_i$ is $\kappa$-distributive,
$\dot{F}_i$ is forced to be a $\kappa$-complete filter. Since
$\kappa$ is forced to be $\theta$-strongly compact in $V[G]^{\mathbb{T}_i}$,
we can fix a $\mathbb{T}_i$-name $\dot{U}_i$ for a $\theta$-complete ultrafilter
over $A$ extending $\dot{F}_i$.

Let $H_0$ be $\mathbb{T}_0$-generic over $V[G]$. As in the proof of
Theorem~\ref{ind_sq_thm}, $H_0$ induces a $\mathbb{T}_i$-generic filter
$H_i$ for each $i < \theta$; let $U_i$ be the realization of
$\dot{U}_i$ in $V[G \ast H_i]$. Note that $W$ remains a normal measure over
$\theta$ in $V[G \ast H_0]$. Define an ultrafilter $U$ on
$(\mathcal{P}(A))^{V[G]}$ in $V[G \ast H_0]$ by setting
\[
X \in U \iff \{i < \theta \mid X \in U_i\} \in W
\]
for all $X \in (\mathcal{P}(A))^{V[G]}$. Since each $U_i$ extends $F$,
$U$ extends $F$ as well. Moreover, exactly as in the argument for the analogous
fact in the proof of Theorem~\ref{ind_sq_thm}, one can show that we in fact
have $U \in V[G]$. It thus remains to show that $U$ is $\theta$-complete.
To this end, fix $\eta < \theta$ and a sequence $\langle X_\xi \mid
\xi < \eta \rangle$ of sets in $U$. Let $X := \bigcap_{\xi < \eta} X_{\xi}$.
Move to $V[G \ast H_0]$. By the definition of $U$ and the
$\theta$-completeness of $W$, there is a set $Y \in W$ such that,
for all $i \in Y$ and all $\xi < \eta$, we have $X_\xi \in U_i$. Then, by
the $\theta$-completeness of $U_i$ for each $i < \theta$, we have $X \in U_i$
for all $i \in Y$. But this implies that $X \in U$, as desired.
\end{proof}

\section{Forcing axioms and indecomposable ultrafilters}\label{section: FAandUltra}

We start by recalling some definitions.

\begin{definition}
Let $M\prec H(\theta)$ and $\delta$ be a cardinal. We say $M$ is \emph{$\delta$-guessing (or is a $\delta$-guessing model)} whenever for any $z\in M$ and any $b\subset z$, if it is the case that for every $a\in M\cap [M]^{<\delta}$ we have $b\cap a\in M$, then there is some $b'\in M$ such that $b'\cap M = b \cap M$. In such a case, we say that $b$ is \emph{$M$-guessed}, and that $b'$ \emph{guesses} $b$.
\end{definition}

\begin{definition}\label{definition: IU}
We say $M\prec H(\theta)$ is $\delta$-internally unbounded if for any $z\in M$ and any $x\in [z\cap M]^{<\delta}$, there is some $y\in [z]^{<\delta} \cap M$ such that $x\subset y$.
\end{definition}

Krueger \cite{krueger_sch} showed that if $M\prec H(\theta)$ is an $\aleph_1$-guessing model, then $M$ is $\aleph_1$-internally unbounded.
\begin{definition}
$\isp(\omega_2)$ asserts that for all large enough $\theta$, the collection $\{M\in \mathcal P_{\aleph_2}(H(\theta))\mid M\prec H(\theta), \ M \text{ is }\aleph_1\text{-guessing}\}$ is stationary in $\mathcal P_{\aleph_2}(H(\theta))$.
\end{definition}

Viale and Weiss \cite{viale_weiss} showed that $\pfa$ implies $\isp(\omega_2)$. Krueger \cite{krueger_sch} showed that $\isp(\omega_2)$ implies $\sch$.

\begin{thm}\label{theorem: ISPMeasurable}
$\isp(\omega_2)$ implies that if $\kappa>2^\omega$ is a cardinal carrying a uniform indecomposable ultrafilter,
then either $\kappa$ is a measurable cardinal or $\kappa$ is the supremum of countably many measurable cardinals.
\end{thm}

First, we record some known constraints regarding $\kappa$ being a successor cardinal.

\begin{fact}[Kunen-Prikry, \cite{kunen_prikry}]
For regular $\lambda$, if an ultrafilter $U$ is $\lambda^+$-decomposable, then it is $\lambda$-decomposable.
\end{fact}

\begin{fact}[Prikry, \cite{prikry}]\label{theorem: prikry}
Suppose that $\lambda$ is a singular strong limit cardinal such that $\lambda^{<\lambda}<2^{\lambda^+}$. Then every uniform ultrafilter $U$ on $\lambda^+$ that is $\beta$-indecomposable for a tail of $\beta<\lambda$ is $\lambda$-decomposable.
\end{fact}

Next, we prove some constraints regarding $\kappa$ not being a strong limit cardinal.

\begin{lemma}\label{lemma: nonstronglimit}
Suppose that $\isp(\omega_2)$ holds, $\kappa$ is a cardinal that is not a strong
limit cardinal, $\lambda < \kappa$ is least such that $2^\lambda \geq \kappa$, and
$\cf(\lambda) > \aleph_0$. Then $\kappa$ does not carry a uniform indecomposable ultrafilter.
\end{lemma}

\begin{proof}
Suppose for the sake of contradiction that $U$ is a uniform indecomposable ultrafilter on
$\kappa$. Let $\vec{f} = \langle f_\alpha \mid \alpha < \kappa \rangle$ be an injective
sequence of elements of ${^{\lambda}}2$. Using the minimality of $\lambda$ and the fact that
$U$ is indecomposable, choose for each $\eta < \lambda$ a countable set $\mathcal{F}_\eta
\subseteq {^{\eta}}2$ such that $\{\alpha < \kappa \mid f_\alpha \restriction \eta
\in \mathcal{F}_\eta\} \in U$. Let $\vec{\mathcal{F}} = \langle \mathcal{F}_\eta \mid
\eta < \lambda \rangle$.

We first handle the case in which $\cf(\lambda)>\aleph_1$. Let $\theta$ be a sufficiently large regular cardinal, and let $N \in\mathcal P_{\aleph_2}(H(\theta))$
be such that $N \prec H(\theta)$, $\{U, \vec{\mathcal{F}}, \vec{f}\} \subseteq N$, and
$N$ is an $\aleph_1$-guessing model. Let $\eta := \sup(N \cap \lambda)$, and note that
$\cf(\eta) = \aleph_1$ (cf.\ \cite[Proposition 2.1(6)]{viale_guessing_models}).
Let $\mathcal{G} := \{g \in \mathcal{F}_\eta \mid g \text{ is $N$-guessed}\}$.

\begin{claim}
$X := \{\alpha < \kappa \mid f_\alpha \restriction \eta \in \mathcal{G}\}$ is in $U$.
\end{claim}

\begin{proof}
Suppose otherwise, and let $\mathcal{H} = \mathcal{F}_\eta \setminus \mathcal{G}$.
Then $Y = \{\alpha < \kappa \mid f_\alpha \restriction \eta \in \mathcal{H}\} \in U$.
For each $h \in \mathcal{H}$, there is $\xi_h \in N \cap \eta$ such that
$h \restriction \xi_h \notin N$; otherwise, $h$ would be $N$-guessed.
Find $\xi \in N \cap \eta$ such that $\xi \geq \xi_h$ for all $h \in \mathcal{H}$.
Then, for all $\alpha \in Y$, we have $f_\alpha \restriction \xi \notin N$.
This contradicts the fact that $\mathcal{F}_\xi \subseteq N$.
\end{proof}

For each $g \in \mathcal{G}$, let $g^* \in N$ be such that $g^* \cap N = g \cap N$.
By elementarity, we have $g^* \in {^\lambda}2$ for all $g \in \mathcal{G}$. By the
$\aleph_1$-internal unboundedness of $N$, we can find a countable set $z \in N$ such that
$g^* \in z$ for all $g \in \mathcal{G}$. We may assume that $z \subseteq {^\lambda}2$.
For each $\xi < \lambda$, let $\mathcal{F}^*_\xi = \{h \restriction \xi \mid
h \in z\}$, and let $X^*_\xi = \{\alpha < \kappa \mid f_\alpha \restriction \xi \in
\mathcal{F}^*_\xi\}$. Then $\langle \mathcal{F}^*_\xi \mid \xi < \lambda \rangle$
and $\langle X^*_\xi \mid \xi < \lambda \rangle$ are in $N$. By elementarity,
$X^*_\xi \in U$ for every $\xi < \lambda$. Moreover, if $\xi < \xi' < \lambda$,
then $X^*_\xi \supseteq X^*_{\xi'}$. Therefore, since $U$ is indecomposable and
$\cf(\lambda) > \aleph_0$,
we have $X^* := \bigcap_{\xi < \lambda} X^*_\xi \in U$. Let
$T = \{h \restriction \xi \mid h \in z, \ \xi < \lambda\}$, so $T \subseteq {^{<\lambda}}2$
is a tree. Then, for every $\alpha \in X^*$, $f_\alpha$ is a cofinal branch
through $T$. However, $T$ is a tree of height $\lambda$ with countable levels,
so, since $\cf(\lambda) > \aleph_1$, $T$ has at most countably many cofinal branches. This is a contradiction.

Assume now that $\cf(\lambda)=\aleph_1$.  Let $\langle \lambda_i\mid i<\omega_1\rangle$ be an increasing sequence of cardinals cofinal in $\lambda$. Consider the tree $T$ of height $\omega_1$ whose $i^{\mathrm{th}}$ level $T_i$ is defined to be $\{l\in {^{\lambda_i}}2 \mid \exists \beta\geq i\ \exists g\in \mathcal{F}_{\lambda_\beta} \text{ such that }l=g\restriction \lambda_i\}$. For each $\xi<\omega_1$, let $X_\xi=\{\alpha<\kappa\mid f_\alpha\restriction \lambda_\xi\in T_\xi\}$. Then we have that for every $\xi<\xi'\in\omega_1$, $X_{\xi}\in U$ and $X_{\xi}\supseteq X_{\xi'}$. Since $U$ is $\omega_1$-indecomposable, $X=\bigcap_{\xi<\omega_1} X_\xi\in U$. In particular, $T$ has $\kappa$ many branches. On the other hand, $T$ has size and height $\omega_1$, so such a tree is a weak Kurepa tree, whose existence contradicts $\isp(\omega_2)$ (see \cite[Theorem 2.8]{CoxKrueger}).
\end{proof}

\begin{remark}
Lemma \ref{lemma: nonstronglimit} is optimal in the following sense: it is consistent that $\isp(\omega_2)$ holds and $2^\omega$ carries a uniform indecomposable ultrafilter. To see this, Cox and Krueger \cite{CoxKrueger} showed that $\isp(\omega_2)$ can be made indestructible under adding any number of Cohen reals.\footnote{It was later shown in \cite{hlhs} that $\isp(\omega_2)$ is in fact \emph{always} indestructible under adding any number of Cohen reals.} Starting with their model and adding measurably many Cohen reals will result in the model as desired.
\end{remark}

\begin{proof}[Proof of Theorem \ref{theorem: ISPMeasurable}]
Suppose that $\isp(\omega_2)$ holds, and suppose that $\kappa > 2^\omega$ carries
a uniform indecomposable ultrafilter $U$. To avoid triviality, assume that $U$ is countably incomplete.
By Lemma \ref{lemma: nonstronglimit}, $\kappa>2^{\omega_1}$ necessarily. So, by Fact~\ref{finest}, we may fix a finest partition $\varphi: \kappa\rightarrow\omega$. Let $D:=\varphi^*(U)$ be the Rudin-Keisler projection. Let $j_U: V\rightarrow M_U$, $j_D: V\rightarrow M_D$ and $k: M_D\rightarrow M_U$ be the corresponding elementary embeddings. Let $W$ be the $M_D$-ultrafilter derived from $k$ and $[\id]_U$. By Theorem~\ref{theorem: silver}, $W$ is $M_D$-$j_D(\mu)$-complete for all $\mu<\kappa$. By Lemma~\ref{lemma: approximation}, for each $z\in\mathcal P_{\omega_2}(\mathcal P(\kappa))$, we have $W\cap j_D(z)\in M_D$.

Let $\theta>2^\kappa$ be a large enough regular cardinal. For each $x\in\mathcal P_{\omega_2}(H(\theta))$ with $x\prec H(\theta)$, let $f_x: \omega\rightarrow \mathcal P(x\cap \mathcal P(\kappa))$ represent $W\cap j_D(x)$ in $M_D$. By the elementarity of $j_D$, we can insist on the following for each $n<\omega$:
\begin{enumerate}
\item $f_x(n)$ is an ultrafilter on $x\cap\mathcal P(\kappa)$;
\item $f_x(n)$ is $x$-$\omega_1$-complete, namely, if $\langle A_i\in f_x(n)\mid i<\omega\rangle \in x$, then $\bigcap_{i<\omega} A_i \in f_x(n)$.
\end{enumerate}
The reason why we can insist on the preceding is that for any $x\in\mathcal P_{\omega_2}(H(\theta))$ with $x\prec H(\theta)$, $M_D\models  W\cap j_D(x)$ is an ultrafilter on $j_D(x\cap\mathcal P(\kappa))$ and for any $\langle A^*_i\in W\cap j_D(x)\mid i<j_D(\omega)\rangle\in j_D(x)$, we have $\bigcap_{i<j_D(\omega)} A^*_i \in W\cap j_D(x)$, since $W$ is $M_D$-$j_D(\omega_1)$-complete.

Let $\theta^*>>\theta$ be a sufficiently large regular cardinal and $M\prec H(\theta^*)$ be an $\aleph_1$-guessing model of size $\aleph_1$ containing all relevant objects. Note that $M$ is internally unbounded by \cite{krueger_sch}. Let $y=M\cap H(\theta)$. For each $x\in\mathcal P_{\omega_2}(H(\theta))\cap M$, we know that the following set is in $D$: $B_x=\{n<\omega\mid f_x(n)=f_y(n)\cap x\}$.

\begin{claim}\label{claim: guessed}
If $m\in \omega$ is such that the set $\{x\in M\cap\mathcal P_{\omega_2}(H(\theta)) \mid m \in B_x\}$
is cofinal in $M\cap\mathcal P_{\omega_2}(H(\theta))$, then $f_y(m)$ is $M$-guessed.
\end{claim}

\begin{proof}
Let $m$ be as in the claim.
For any $x\in M\cap [M]^\omega$, we need to show that $f_y(m)\cap x\in M$. We may assume $x\subset P(\kappa)$. By the hypothesis, there is some $x'\in M\cap\mathcal P_{\omega_2}(H(\theta))$ containing $x$ such that $m \in B_{x'}$. As a result, $f_{x'}(m)=f_y(m)\cap x'$. But then $f_y(m)\cap x = f_{x'}(m)\cap x \in M$.
\end{proof}

Note that if $m$ is as in the claim and $a \in M$ guesses $f_y(m)$, then the elementarity of $M$ and the fact that $f_y(m)$ is $y$-$\omega_1$-complete imply that $a$ is a $\sigma$-complete ultrafilter on $\kappa$.

\begin{claim}
$X:=\{i\in \omega\mid f_y(i) \text{ is $M$-guessed }\}$ is in $D$.
\end{claim}
\begin{proof}
Suppose not for the sake of contradiction. For each $i\in \omega \setminus X$,
Claim \ref{claim: guessed} implies that we can fix $z_i \in M\cap\mathcal P_{\aleph_2}(H(\theta))$
such that, for all $z \in  M\cap\mathcal P_{\aleph_2}(H(\theta))$ containing $z_i$, we have
$i \notin B_{z}$. By the $\aleph_1$-internal unboundedness of $N$, there is some $z^*\in  M\cap\mathcal P_{\aleph_2}(H(\theta))$ such that $z_i\subset z^*$ for all $i\in \omega \setminus X$. But then
$i \notin B_{z^*}$ for all $i \in \omega \setminus X$, contradicting the fact that $B_{z^*} \in D$.
\end{proof}

As a result, for each $\aleph_1$-guessing model $N \prec H(\theta^*)$, there is a set $X_N \in D$ such that for every $i\in X_N$, $f_{N\cap H(\theta)}(i)$ is $\aleph_1$-guessed, hence there is some $\sigma$-complete ultrafilter $V_i=V^N_i\in N$ on $\kappa$ such that $V_i\cap N = f_{N\cap H(\theta)}(i) \cap N$.
If $\cf(\kappa)>\omega$ and there is some $\aleph_1$-guessing model $N$ with one of $V_i^N$ being $\kappa$-complete, then $\kappa$ is measurable. If $\cf(\kappa)=\omega$ and for any $\mu<\kappa$ there is an $\aleph_1$-guessing model $N$ with $V_i^N$ being $\mu$-complete, then $\kappa$ is a supremum of countably many measurable cardinals. Suppose the situation above does not occur for the sake of contradiction.

Define $$\delta:=\begin{cases}
\kappa,&\text{if }\cf(\kappa)>\omega;\\
\mu,&\text{if }\cf(\kappa)=\omega,
\end{cases}$$
where $\mu<\kappa$ is some cardinal such that for any $\aleph_1$-guessing model $N$ and $i\in X_N$, $V_i^N$ is not $\mu$-complete.

By the $\aleph_1$-internal unboundedness of $N$, there exists some $z_N\in N\cap\mathcal P_{\aleph_1}(N)$ such that $V^N_i\in z_N$ for all $i\in X_N$. Apply the pressing down lemma to find  stationary $S\s\mathcal P_{\aleph_2} (H(\theta^*))$ consisting of $\aleph_1$-guessing models and a $z$ such that, for all $N\in S$, $z_N=z$. We may also assume that every $a\in z$ is a $\sigma$-complete ultrafilter on $\kappa$. Fix $a\in z$. Let its completeness be $\gamma_a$. Then we can find a $\supseteq$-decreasing sequence $\vec{A}^a=\langle A^a_i\in a\mid i<\gamma_a\rangle$ such that $\bigcap_{i<\gamma_a} A^a_i=\emptyset$. Note that, necessarily, $\gamma_a$ is a measurable cardinal below $\delta$.

Let $E=\{A^a_{i}\mid a\in z, \ i<\gamma_a\}$ and $\gamma^*=\sup_{a\in z} \gamma_a$. Note that by our assumption $\gamma^*<\kappa$: this is clear when $\cf(\kappa)>\omega$ and when $\cf(\kappa)=\omega$, our assumption implies that each $\gamma_a<\delta$ for all $a\in z$.
Since $|E|\leq \gamma^*$, we know that $W\cap j_D(E)\in M_D$. To see this, note that $\gamma^*$ is
either a singular cardinal of countable cofinality or a measurable cardinal. If $\gamma^*$ is a singular cardinal of countable cofinality, then $\gamma^*$ is a strong limit cardinal. By \cite{krueger_sch}, $\isp(\omega_2)$ implies $\sch$. Hence $2^{\gamma^*}=(\gamma^*)^+$. Theorem \ref{theorem: prikry} implies that $\kappa>(\gamma^*)^+$. As a result, Lemma~\ref{lemma: approximation} implies that $j_D(E)\cap W\in M_D$. If $\gamma^*$ is a measurable cardinal $<\kappa$, then Lemma~\ref{lemma: nonstronglimit} implies that $2^{\gamma^*}<\kappa$. We can then apply Lemma~\ref{lemma: approximation} to get the same conclusion as desired.

Let $l: \omega\rightarrow V$ represent $W\cap j_D(E)$ in $M_D$. Let $\mathcal{F}=\langle \vec{A}^a\mid a\in z\rangle$. Then in $M_D$, it is true that for each $\vec{B}\in j(\mathcal{F})$, there exists some $i<lh(\vec{B})$ such that $B_j\not\in W\cap j_D(E)$ for all $j>i$. Here we are using the fact that $\vec{B}\subset j_D(E)$, $lh(\vec{B})\leq j_D(\gamma^*)$ and $W$ is $M_D$-$j_D((\gamma^*)^+)$-complete. By \L o\'s' theorem, there is a set $A\in D$ such that for each $i\in A$ and $a\in z$, there is some $j_{a,i}<\gamma_a$ such that for every $k>j_{a,i}$, it is the case that $A^a_k\not\in l(i)$. By adjusting $l$ if necessary,
we may assume that $A=\omega$ for simplicity.
Therefore, for each $a\in z$, we can find $j_a=\sup_{i\in \omega}j_{a,i}+1<\gamma_a$ (recall that $\gamma_a$ is measurable) such that $A^a_{j_a}\not\in l(i)$ for all $i\in \omega$.

Finally, consider $L=\{A^a_{j_a}\mid a\in z\}$. Let $N\in S$ be such that $L\in N$. In $M_D$, there is some $a^*\in j_D(z)$ such that $a^*\cap j_D(N) = W\cap j_D(N)$. Let $p: \omega\rightarrow z$ represent $a^*$. Consider $q: \omega\rightarrow L$ such that $q(i)=A^{p(i)}_{j_{p(i)}}$ for every $i\in \omega$. Then $[q]_D\in j_D(L)\cap a^* \cap j_D(N)\s j_D(L)\cap W$. However, by our choice of $q$, we have $[q]_D\not\in [l]_D=W\cap j_D(E)\supseteq W\cap j_D(L)$. This is a contradiction.
\end{proof}

\begin{proof}[Proof of Theorem \ref{theorem: main1}]
$\pfa$ implies $2^\omega=\omega_2$ and $\isp(\omega_2)$. Apply Theorem \ref{theorem: ISPMeasurable}.
\end{proof}

Finally, as promised in the introduction, we make use of all the combinatorial analyses in the proceeding sections to show that Theorem \ref{theorem: main1} is optimal.

\begin{thm}\label{theorem: optimal}
$\mm$ is consistent with the existence of a non weakly compact strongly inaccessible cardinal $\kappa$ carrying a uniform $[\aleph_2, \kappa)$-indecomposable ultrafilter.
\end{thm}

\begin{proof}
Start with a model of $\mm$ that contains a measurable cardinal $\kappa$ whose
measurability is indestructible under $\mathrm{Add}(\kappa,1)$.
Let $\bb{P} := \bb{P}^-(\kappa, \omega_1)$ be the forcing from
Definition~\ref{definition: ForcingIndexedMinus} for adding a witness
to $\inde(\kappa, \omega_1)$, and let $G$ be $\bb{P}$-generic over
$V$. By Lemma~\ref{lemma: directedclosed},
$\bb{P}$ is $\omega_2$-directed closed, so $\mm$ holds in $V[G]$.
Moreover, $\inde(\kappa, \omega_1)$ holds in $V[G]$ so, by Proposition
\ref{prop: inde_wc}, $\kappa$ is not weakly compact in $V[G]$.

It remains to show that $\kappa$ carries a uniform $[\aleph_2, \kappa)$-indecomposable
ultrafilter in $V[G]$. The proof of this fact is almost identical to that of
Theorem \ref{ind_sq_thm}, so we only provide a few details, leaving the rest to
the reader.

In $V[G]$, let $\vec{{C}} = \bigcup G = \langle C_{\alpha,i} \mid \alpha \in
\acc(\kappa), ~ i(\alpha) \leq i < \omega_1 \rangle$ be the generically added
witness to $\inde(\kappa, \omega_1)$. For each $i < \omega_1$, define a
poset $\bb{T}_i$ as follows. The underlying set of $\bb{T}_i$ is
$\{C_{\alpha,i} \mid \alpha \in \acc(\kappa) \text{ and } i(\alpha) \leq i\}$.
Given $C_{\alpha,i}, C_{\beta,i} \in \bb{T}_i$, we set
$C_{\beta,i} \leq_{\bb{T}_i} C_{\alpha,i}$ if and only if $\alpha \leq \beta$
and, for all $i \leq j < \omega_1$, we have $C_{\alpha,j} = C_{\beta,j} \cap
\alpha$. The following facts are proven exactly as in
\cite[Lemma 3.18]{hlh} and Proposition \ref{compatible_prop} above.
\begin{itemize}
\item In $V$, for all $i < \omega_1$, the two-step iteration $\bb{P} \ast
\dot{\bb{T}}_i$ has a dense $\kappa$-directed closed subset of cardinality
$\kappa$.
\item In $V[G]$, there is a system of commuting projections
$\langle \pi_{ij} : \bb{T}_i \rightarrow \bb{T}_j \mid i \leq j <
\omega_1 \rangle$ defined by letting $\pi_{ij}(C_{\alpha,i}) = C_{\alpha,j}$
for all $i \leq j < \omega_1$ and $C_{\alpha,i} \in \bb{T}_i$.
\item Suppose that $i_0 \leq i_1 < \omega_1$, $t_0 \in \bb{T}_{i_0}$, and
$t_1 \in \bb{T}_{i_1}$. Then, for all sufficiently large $j < \theta$,
the conditions $\pi_{i_0j}(t_0)$ and $\pi_{i_1j}(t_1)$ are compatible
in $\bb{T}_j$.
\end{itemize}

Let $W$ be a uniform ultrafilter on $\omega_1$, and use it, together with the
fact that, for each $i < \omega_1$, $\kappa$ is measurable in the extension of
$V[G]$ by $\bb{T}_i$, to define an ultrafilter $U$ on $\kappa$ as in the proof
of Theorem \ref{ind_sq_thm}. The verification that $U$ is a uniform,
$[\aleph_2, \kappa)$-indecomposable ultrafilter is as in the proof of
Theorem \ref{ind_sq_thm}, so we leave it to the reader.
\end{proof}

\section{Open questions}\label{Section: questions}

\begin{q}
Does $\pid$ or $\mrp$ imply that any strong limit cardinal carrying a uniform indecomposable ultrafilter is either measurable or a supremum of countably many measurable cardinals?
\end{q}

\begin{q} Does $\ssr$ refute $\square(\kappa,\omega_1)$ for regular $\kappa > \omega_2$?
\end{q}

\section*{Acknowledgments}

The first author was was supported by GA\v{C}R project 23-04683S and the Czech Academy of Sciences (RVO 67985840).
The second author was partially supported by the European Research Council (grant agreement ERC-2018-StG 802756)
and by the Israel Science Foundation (grant agreement 203/22). A portion of this work was carried out while
the authors were participating in the Thematic Program on Set Theoretic Methods in
Algebra, Dynamics and Geometry at the Fields Institute in the spring of 2023.
We thank the Fields Institute for their support and hospitality.

\end{document}